\documentclass[12pt]{article}
\usepackage{amssymb, amsmath}
\usepackage{subcaption}
\usepackage[usenames,dvipsnames]{color}
\usepackage{amsthm}
\usepackage{mathrsfs}
\usepackage{hyperref}
\usepackage[margin=1 in]{geometry}
\usepackage{enumerate}
\allowdisplaybreaks[1]
\numberwithin{equation}{section}

\newtheorem{example}{Example}[section]
\newtheorem{thm}{Theorem}[section]
\newtheorem{obs}{Observation}[section]
\newtheorem{cor}{Corollary}[section]

\newtheorem{pro}{Proposition}[section]
\newtheorem{lemma}{Lemma}[section]
\usepackage{graphicx}
\begin{document}
	\markboth{R. Rajkumar and T. Anitha}{}
	\title{Automorphism group of the reduced power (di)graph of a finite group}
	
	
	\author {T. Anitha\footnote{e-mail: {\tt tanitha.maths@gmail.com},} \footnote{Author is supported by V.V.V College Managing Board Research Grant under VVVCMB-MRP Scheme 2023} 
		\\ \small \it Department of Mathematics,
		\small \it V.V.Vanniaperumal College for Women,\\
		\small \it Virudhunagar--626 001, Tamil Nadu, India.\\
		\and 
		R. Rajkumar\footnote{e-mail: {\tt rrajmaths@yahoo.co.in}} 
		\\ \small \it Department of Mathematics,
		\small \it The Gandhigram Rural Institute (Deemed to be University),\\
		\small \it Gandhigram--624 302, Tamil Nadu, India.\\
	}
	\date{}
	\maketitle
	
	\begin{abstract}
		
	We describe the full automorphism group of the directed reduced power graph and the undirected reduced power graph of a finite group. We compute the full automorphism groups of these graphs of several classes of finite groups. Also, we establish some relation between the automorphism group of the undirected reduced power graph (resp. the directed reduced power graph) and the automorphism group of the undirected power graph   (resp. the directed power graph) of a finite group.

		\paragraph{Keywords:}Group, Power graph, Reduced power graph, Automorphism group.
		\paragraph{Mathematics Subject Classification:}05C25, 20B25.
	\end{abstract}
	
	\section{Introduction}
The \textit{directed power graph} $\overrightarrow{\mathcal{P}}(G)$  of a group $G$ is the digraph with vertex set $G$, in which there is an arc from $u$ to  $v$ if and only if $u \neq v$ and $v$ is a power of $u$; or equivalently $\left\langle v\right\rangle \subseteq \left\langle u \right\rangle $. 	The \textit{(undirected) power graph $\mathcal{P}(G)$ of $G$} is the underlying graph of $\overrightarrow{\mathcal{P}}(G)$. Kelarev and Quinn introduced the directed power graph of a semigroup~\cite{Kelarev} and a group~\cite{Kelarev 2000}. Chakrabarty et al. defined the power graph of a semigroup~\cite{Chakrabarty}. 		
The power graph has been studied extensively by many authors; see for instance~\cite{AliD,cameron,peter, Curtin,min feng 2015,Kelarev 2001,Mehranian} and a survey~\cite{survey}. The computation of the automorphism group of the power graph of a cyclic group was initiated by Alireza et al.~\cite{AliD} and settled by Mehranian et al.~\cite{Mehranian}. The automorphism group of the power graph of the dihedral group was also computed in~\cite{Mehranian}. Min Feng et al.~\cite{Min Aut PG 2016} described the full automorphism group of $\overrightarrow{\mathcal{P}}(G)$ and $\mathcal{P}(G)$  for a  finite group $G$. By using these, the computation of the automorphism group of $\overrightarrow{\mathcal{P}}(G)$ and $\mathcal{P}(G)$ when $G$ is cyclic, elementary abelian, dihedral
and generalized quaternion groups have been made. Ali Reza Ashrafi et al.~\cite{Ali} computed the
automorphism group of the power graph of some more classes of finite groups.

In \cite{raj RPG 2017}, the authors defined the following: The \textit{directed reduced power graph} $\overrightarrow{\mathcal{RP}}(G)$ of a group $G$ is the digraph with the elements of $G$ as its vertices, and for two vertices $u$ and $v$ in $G$, there is an arc from $u$ to $v$ if and only if  $v$ is a power of $u$ and $\left\langle v\right\rangle \neq \left\langle u\right\rangle $; or equivalently $\left\langle v\right\rangle \subset \left\langle u \right\rangle $. The \textit{(undirected) reduced power graph $\mathcal{RP}(G)$ of $G$} is the underlying graph of $\overrightarrow{\mathcal{RP}}(G)$. 	
For the study on the reduced power graphs, we refer the reader to ~\cite{raj RPG and PG 2018,raj RPG embedd 2018, Li 2021, Ma genus 3, Xma strong 2021, Xma 2021,raj RPG 2018,raj RPG LS 2019}.

For simplicity, in this paper we call ``directed reduced power graph (resp. directed power graph)" as ``reduced power digraph (resp. power digraph)". 

The first aim of this paper is to describe the full automorphism group of the  reduced power (di)graph of a finite group.  
In achieving this, we could see that it requires significantly distinct arguments based on the structure of the reduced power (di)graph in comparison with the arguments used for computing the automorphism group of the  power (di)graph of a finite group (cf. \cite{Min Aut PG 2016}).
 The second aim of this paper is to describe some relations between the automorphism group of the reduced power graph (resp. the reduced power digraph) and the automorphism group of the power graph   (resp. the power digraph) of a finite group.

Let $D$ be a  digraph. The arc set of $D$ and the automorphism group of $D$ are denoted by $A(D)$ and $Aut(D)$, respectively. The in-degree, the out degree, the set of all in-neighbors and the set of all out-neighbors of a vertex $v$ in $D$ are denoted by $deg_{D}^{-}(v)$, $deg_{D}^{+}(v)$, $N_{D}^-{(v)}$ and $N_{D}^+{(v)}$, respectively. For a vertex $v$ in an undirected graph $\Gamma$, the degree of $v$ and the set of all neighbors of $v$ are denoted by $deg_{\Gamma}(v)$ and $N_{\Gamma}(v)$, respectively.

Let  $H \wr K$ denote the wreath product of the groups $H$ over $K$. For an integer $n\geq 1$, $\mathbb Z_n$ denotes the additive group of integers modulo $n$. For integers $m, n \geq 1$, $\mathbb Z_n^m$ denotes the direct product of $m$ copies of $\mathbb Z_n$. The dihedral group of order $2n$ ($n\geq 3$) is given by $D_{2n}=\left\langle a,b \mid a^n=e=b^2,\right.$ $\left. ab=ba^{-1} \right\rangle $. \noindent The generalized quaternion group of order $4n$ ($n\geq 2$) is given by $Q_{4n}=\left\langle a, b\mid  a^{2n}=e=b^4, bab^{-1}=a^{-1} \right\rangle $. \noindent The semi-dihedral group of order $8n$ ($n\geq 2$) is given by $SD_{8n}=\left\langle a, b\mid a^{4n}=e =b^2, bab^{-1}=\right.$ $\left.a^{2n-1} \right\rangle $. The modular group of order $p^\alpha$, $p$ a prime and $\alpha\geq 3$ is given by $\mathbb M_{p^\alpha}=\left\langle a,b \mid a^{p^{\alpha-1}}=e=b^p,\right.$ $\left. bab^{-1}=a^{p^{\alpha-2}+1} \right\rangle $.  For $n\geq 1$, the group $V_{8n}$ of order $8n$ is defined as $V_{8n}=\left\langle  a,b\mid a^{2n}=b^4=e, aba=b^{-1}, ab^{-1}a=b\right\rangle $ and the group $U_{6n}$ of order $6n$ is defined as $U_{6n}:=\left\langle  a,b\mid a^{2n}=b^3=e, b^{-1}ab=a^{-1}\right\rangle $.  The identity map on the set $\Omega$ is denoted by $\textbf{1}_{\Omega}$.

The rest of this paper is arranged as follows. In Section~2, we give some basic constructions for determining the automorphism group of the reduced power (di)graph of a finite group and to introduce a faithful group action on that group. In Section~3, we establish some results on the induced action of the automorphism group of the reduced power (di)graph  on the set of all equivalence classes of an equivalence relation defined in Section~2. In Section~4, we give the full structure of $Aut(\overrightarrow{\mathcal{RP}}(G))$ and
$Aut\left(\mathcal{RP}(G)\right)$ for a finite group $G$. In Section~5, we compute the full automorphism group of of the reduced power (di)graph of several classes of finite groups such as cyclic group, dihedral group, generalized quaternion group, semi-dihedral group, the group $V_{8n}$, the group $U_{6n}$, $p$-group with exponent $p$ or non-nilpotent group of order $p^mq$ with all non-trivial elements are of order $p$ or $q$, where $p$, $q$ are distinct primes. In Section~6, we establish the results, which describe some relation between the automorphism group of the  reduced power graph (resp. the  reduced power digraph) and the automorphism group of the  power graph   (resp. the  power digraph) of a finite group.	

\section{Two faithful group actions on a group}

In the rest of the paper $G$ denotes a  finite group. Let $\mathscr{C}(G)$ be the set of all cyclic subgroups of $G$. For $C\in \mathscr{C}(G)$, let $[C]$ denote the set of all generators of $C$. Take $\mathscr{C}(G)=\{C_1,C_2,\ldots,C_k\}$ and $[C_i]=\{[C_i]_1,[C_i]_2,\ldots,[C_i]_{s_i}\}$. Notice that $G$ is the disjoint union of $[C_1],[C_2],\ldots,[C_k]$.

Let $P(G)$ be the set of all permutations on $\mathscr{C}(G)$ preserving the order, inclusion and non-inclusion with the condition that if the image of $C$ is different from $C$, then it has at least one cyclic subgroup different from the subgroup of $C$. That is, for each $\sigma\in P(G)$, $\mid C_i^\sigma\mid=\mid C_i\mid$ for every $i\in \{1,2,\ldots,k\}$;  $C_i\subseteq C_j$ if and only if $C_i^\sigma\subseteq C_j^\sigma$ and either $C_i^\sigma=C_i$ or  the set of all cyclic subgroups of $C_i$ and the set of all cyclic subgroups of $C_i^\sigma$ are different. Let $M(G)$ be the set of all maximal cyclic subgroup of $G$. Take $M(G)=\{C_{i_1}, C_{i_2},\ldots,C_{i_r}\}$, where $C_{i_t}=\left\langle x_t\right\rangle$ for $t=1,2,\ldots,r$. Let $\mathscr{M}(G)$ be the restriction map of $P(G)$ on the set $M(G)$. Then $\mathscr{M}(G)$ is a permutation group on $M(G)$.
This group induces a faithful action on the set G:	
\begin{align}\label{cyclic group permutation}
	\mathscr{M}(G) \times G  \rightarrow G, \hspace{0.5cm} (\sigma, {x_i}^j)\mapsto {(x_i^\sigma)}^j.
\end{align}

Note that for each $x\in G$, $x\in \left\langle x_i\right\rangle $ for some $i=1,2,\ldots,r$ and so $x=x_i^j$. Suppose $x\in \left\langle x_i\right\rangle $ and $\left\langle x_j\right\rangle $, where $i\neq j$ and $i,j\in \{1,2,\ldots,r\}$. Then $x=x_i^s$ and $x=x_j^t$. Since $x^\sigma \in \left\langle x_i^\sigma\right\rangle $, $\left\langle x_j^\sigma \right\rangle $ and $o(x)=o(x^\sigma)$, so $\left\langle (x_i^\sigma)^s\right\rangle = \left\langle (x_j^\sigma)^t\right\rangle$. So for each $x\in G$, we fix any one maximal cyclic subgroup $C_{i_j}=\left\langle x_j\right\rangle $ which contains $x$ and write all the elements of $\left\langle x\right\rangle $ as a power of $x_j$.

Now consider the equivalence relation $\simeq$ defined on $G$ as follows~\cite{raj RPG and PG 2018}: for $x,y\in G$, $x\simeq y$ if and only if $N_{\mathcal{RP}(G)}(x)=N_{\mathcal{RP}(G)}(y)$. It is shown in~\cite[Proposition 2.2]{raj RPG and PG 2018} that $x\simeq y$ if and only if $N^-_{\overrightarrow{\mathcal{RP}}(G)}(x)=N^-_{\overrightarrow{\mathcal{RP}}(G)}(y)$ and $N^+_{\overrightarrow{\mathcal{RP}}(G)}(x)=N^+_{\overrightarrow{\mathcal{RP}}(G)}(y)$.  Let $\widehat{x}$ denote the $\simeq$-class determined by $x$.

Let $\mathscr{R}(G)=\{\widehat{x}\mid x\in G\}=\{\widehat{u_1},\widehat{u_2},\ldots,\widehat{u_t}\}$. Notice that $G$ is the disjoint union of $\widehat{u_1},\widehat{u_2},\ldots,\widehat{u_t}$. For $\Omega \subseteq G$, let $S_{\Omega}$ denote the symmetric group acting on the set $\Omega$. To prove our main result we need the following faithful group action on $G$ which is
analogous to \cite[Eqn. (2)]{Min Aut PG 2016}.
 	
\begin{equation}\label{action2}
	\left( {\prod_{i=1}^{t}S_{\widehat{u_i}}} \right) \times G   \rightarrow G, \hspace{0.5cm} ((\tau_1,\ldots,\tau_t),x )\mapsto x^{\tau_i}, \mbox{~~~where~~} x \in \widehat{u_i}.
\end{equation}

\section{The induced action of $Aut(\overrightarrow{\mathcal{RP}}(G))$ and $Aut({\mathcal{RP}}(G))$ on $\mathscr{R}(G)$}

Consider the equivalence relation $\sim$ defined on $G$ as follows: for
$x$, $y\in G$, $x \sim y$ if and only if $\left\langle x\right\rangle =\left\langle y\right\rangle $. Let $[x]$ denote the  $\sim$-class determined by $x$.


It is shown in~\cite[Lemma 2.3]{raj RPG and PG 2018} that any $\simeq$-class $\widehat{x}$ can be written as the union of distinct $\sim$-classes $[x_1]$, $[x_2]$, $\ldots, [x_r]$.  Also  $\widehat{x}$  is said to be of

\begin{itemize}
	\item[--] \textbf{Type~I}, if $r=1$, that is $\widehat{x}=[x]$;
	
	\item[--] \textbf{Type~II}, if $r\geq 2$ and $\left\langle x_i\right\rangle $s are distinct maximal cyclic subgroups of $G$ of same order for  all $i=1,2,\ldots,r$. We call $(o(x),r)$ as the parameter of $\widehat{x}$; 
	
	\item[--] \textbf{Type~III}, if $G$ is non-cyclic with $r\geq 2$ and the $\left\langle x_i\right\rangle $s are maximal cyclic subgroups of prime orders in $G$ such that $\mid \{o(x_i):~i=1,2,\ldots,r\}\mid \geq 2$; 
	
	\item[--] \textbf{Type~IV}, if one of the following holds:
	\begin{enumerate}[(i)]
		\item $G$ is non-cyclic with $r=2$ and $\left\langle x_1\right\rangle $, $\left\langle x_2\right\rangle $ are non-maximal cyclic subgroups of distinct prime orders in $G$;
		
		\item $G$ is cyclic with $r=2$ and $\left\langle x_1\right\rangle $, $\left\langle x_2\right\rangle $ are maximal cyclic subgroups of distinct prime orders in $G$.
		
		In this case we say the pair $(o(x_1),o(x_2))$ as the parameter of $\widehat{x}$.
	\end{enumerate}
\end{itemize}

From the above classification, we get the following

\begin{obs}\label{ob1}
	For any $x\in G$, the following hold.
	\begin{enumerate}[\normalfont(1)]
		\item If  $\widehat{x}$ is of Type~I, then $\mid \widehat{x} \mid =\varphi(o(x))$;
		
		\item If  $\widehat{x}$ is of Type~II with parameter $(m,r)$, then $\mid \widehat{x}\mid =r\varphi(m) $;
		
		\item If  $\widehat{x}$ is of Type~IV with parameter $(p,q)$, then $\mid \widehat{x}\mid =p+q-2$.
	\end{enumerate}
\end{obs}

\begin{obs}\label{ob2}
	\begin{enumerate}[\normalfont(1)]
		\item  Let $x$ and $y$ be two non-trivial elements in $G$ such that $\left\langle y \right\rangle \subset \left\langle x\right\rangle $. Then $\widehat{x}$ is either of Type~I or of Type~II and $\widehat{y}$ is either of Type~I or of Type~IV.\label{pres of inclusion}
		
		\item Let $x,y \in G$. If $\widehat{x}$ and $\widehat{y}$  are of Type~III, then $\widehat{x}=\widehat{y}$. \label{Type iii}
	\end{enumerate}
\end{obs}


	%

It is easy to see that for each $x \in G$ and $\pi \in Aut(\overrightarrow{\mathcal{RP}}(G))$ (resp. $\pi\in Aut\left(\mathcal{RP}(G)\right)$), we have $\widehat{x}^\pi=\widehat{x^\pi}$. Hence $Aut(\overrightarrow{\mathcal{RP}}(G))$ (resp. $Aut(\mathcal{RP}(G))$) induces an action on $\mathscr{R}(G)$ as follows:
\begin{equation}\label{action3}
	Aut(\overrightarrow{\mathcal{RP}}(G)) \times\mathscr{R}(G)  \rightarrow \mathscr{R}(G),\hspace{0.5cm} (\pi, \widehat{x}) \mapsto \widehat{x^\pi}
\end{equation}
and
\begin{equation}\label{action4}
	Aut(\mathcal{RP}(G)) \times \mathscr{R}(G)  \rightarrow \mathscr{R}(G),\hspace{0.5cm} (\pi, \widehat{x}) \mapsto \widehat{x^\pi}.
\end{equation}

Next we shall show that each orbit of $Aut(\mathcal{RP}(G))$ on $\mathscr{R}(G)$ consists of some equivalence classes of the same type. 

\begin{lemma}\label{Auto lemma2}
	Let $p$, $q$ be distinct primes. Let $x$ and $y$ be two non-trivial elements in $G$ such that $\left\langle y \right\rangle \subset \left\langle x\right\rangle $. Then we have the following.
	\begin{enumerate}[\normalfont(1)]
		\item If $\widehat{x}$ is of Type~I, then one of the following holds:
		
		\begin{enumerate}[\normalfont(a)]
			\item $\mid \widehat{y}\mid \leq \mid \widehat{x}\mid $, if $\widehat{y}$ is either of Type~I or of Type~IV with parameter $(p,q)$, where $p,q\neq 2$. Equality holds if both $\widehat{x}$ and $\widehat{y}$ are of Type~I with $o(x)=2 \cdot o(y)$ and $o(y)$ is odd.
			
			\item $\mid \widehat{y}\mid >\mid \widehat{x}\mid $, if $\widehat{y}$ is of Type~IV with parameter $(2,p)$ and $o(x)=2p$, where $p$ is a prime.
		\end{enumerate}
		
		\item If $\widehat{x}$  is of Type~II, then $\mid \widehat{y}\mid <\mid \widehat{x}\mid $. 
	\end{enumerate}
\end{lemma}

\begin{proof}
	By Observation~\ref{ob2}(1), $\widehat{x}$ is either of Type~I or of Type~II, and $\widehat{y}$ is either of Type~I or of Type~IV. 
	
	(1) Suppose that $\widehat{x}$ is either of Type~I. 
	
	If $\widehat{y}$ is of Type~I, then $\mid \widehat{y}\mid =\varphi(o(y))$ and also $\mid \widehat{x}\mid =r\cdot \varphi(o(x))$, where $r\geq 1$. It is known that for any two positive integers $m$ and $n$, if $m\mid n$, then $\varphi(m)\leq \varphi(n)$; equality holds if either $n=m$ or $m$ is odd and $n=2m$.  Since $o(y)\mid o(x)$, it follows that $\varphi(o(y))\leq \varphi(o(x))$. So $\mid \widehat{y}\mid =\varphi(o(y))\leq \varphi(o(x))\leq r\cdot \varphi(o(x))= \mid \widehat{x}\mid $; equality holds if $r=1$, $o(x)=2\cdot o(y)$ and $o(y)$ is odd, since $o(x)\neq o(y)$.
	
	If $\widehat{y}$ is of Type~IV with parameter $(p,q)$, then $\mid \widehat{y}\mid =p+q-2$. If $p,q\neq 2$, then $p+q-2<(p-1)(q-1)$. So $\mid \widehat{y}\mid <\mid \widehat{x}\mid $. Suppose $q=2$. Then $\mid \widehat{y}\mid =p$. If $o(x)=2p$ and $r=1$, then $\mid \widehat{x}\mid =(p-1)<p=\mid \widehat{y}\mid $. If $o(x)=2p$ and $r> 1$, then $\mid \widehat{x}\mid=r\cdot (p-1) > p=\mid \widehat{y}\mid $. If $o(x)\neq 2p$, then $\mid \widehat{x}\mid > p=\mid \widehat{y}\mid $. 
	
	(2) If $\widehat{x}$ is of Type~II with parameter $(m,r)$, then $\widehat{y}$ is of Type~I. By parts (1) and (2) of Observation~\ref{ob1}, $\mid \widehat{x}\mid=r.\varphi(m)$ and $\mid \widehat{y}\mid=\varphi(o(y))\leq \varphi(m)$.  So $\mid \widehat{y}\mid <\mid \widehat{x}\mid $ as desired.
\end{proof}

\begin{lemma}(\cite[Lemma 2.1]{raj RPG and PG 2018})\label{ZQ}
	Let $e$ be the identity element of $G$. Let $S$ be the set of all vertices of $\mathcal{RP}(G)$ which are adjacent to all other vertices of $\mathcal{RP}(G)$.Then either $S=\{e\}$ or $G = \mathbb Z_{2^m}$$( m \geq 1)$ or $Q_{2^\alpha}$$(\alpha \geq 3)$ and $S=\{e,a\}$, where $a$ is the unique element of order $2$.
\end{lemma}

\begin{lemma}\label{image of e}
	\begin{enumerate}[\normalfont(1)]
		\item Let $\pi\in Aut({\mathcal{RP}}(G))$. Then we have the following. \label{Auto thm 1}
		\begin{enumerate}[\normalfont(i)]
			\item If $G\ncong \mathbb Z_{2^m}(m\geq 1)$ and $Q_{2^\alpha}(\alpha\geq 3)$, then $\pi$ fixes $e$.
			
			\item If $G\cong \mathbb Z_{2^m}(m\geq 1)$ or $Q_{2^\alpha}(\alpha\geq 3)$, then $e^\pi=e$ or $a$, where $a$ is the unique element of order $2$ in $G$.
		\end{enumerate}
		\item If $\pi\in Aut( \overrightarrow{\mathcal{RP}}(G)) $, then $\pi$ fixes $e$. \label{auto of RP fix e}
	\end{enumerate}
\end{lemma}

\begin{proof}
	(1) Notice that $deg_{\mathcal{RP}(G)}(e)=deg_{\mathcal{RP}(G)}(e^\pi)$. Also, $deg_{\mathcal{RP}(G)}(e)=|G|-1$. So $e^\pi$ is adjacent to all the vertices of $\mathcal{RP}(G)$. Thus the result follows from Lemma~\ref{ZQ}.
	
	(2) Since $deg_{\overrightarrow{\mathcal{RP}}(G)}^+(e)=0$ and $deg_{\overrightarrow{\mathcal{RP}}(G)}^+(x)\geq 1$ for every $(e\neq ) x\in G$, so $e^\pi=e$.
\end{proof}

\begin{lemma}(\cite[In the proof of Theorem 3.2]{raj RPG and PG 2018})\label{relation of 2 eql class}
	Let $x$ and $y$ be two non-trivial elements in $G$ such that they are adjacent in $\mathcal{RP}(G)$. Then one of the following holds:
	\begin{enumerate}[\normalfont(1)]
		\item    $\mid \widehat{x}\mid >\mid \widehat{y}\mid $ with $\mid \widehat{x}\mid $ is even and $\left\langle y \right\rangle \subset \left\langle x \right\rangle $,
		
		\item  $\mid \widehat{x}\mid >\mid \widehat{y}\mid $ with $\mid \widehat{x}\mid $ is odd and $\left\langle x \right\rangle \subset \left\langle y \right\rangle $,
		\item $\mid \widehat{x}\mid =\mid \widehat{y}\mid $ with either $deg_{\mathcal{RP}(G)}(x)$ is even, $deg_{\mathcal{RP}(G)}(y)$ is odd and $\left\langle y \right\rangle \subset \left\langle x \right\rangle $, or $deg_{\mathcal{RP}(G)}(x)$ is odd, $deg_{\mathcal{RP}(G)}(y)$ is even and $\left\langle x \right\rangle \subset \left\langle y \right\rangle $.
	\end{enumerate}
\end{lemma}

\begin{lemma}\label{Auto presetve suset}
	
	\begin{enumerate}[\normalfont(1)]
		\item Let $\pi\in Aut(\overrightarrow{{\mathcal{RP}}}(G))$ and let $x,y\in G$. Then $\left\langle y \right\rangle \subset \left\langle x\right\rangle $ if and only if $\left\langle y^\pi\right\rangle \subset \left\langle x^\pi\right\rangle $.
		
		\item Let $\pi\in Aut\left(\mathcal{RP}(G)\right)$ and let $x$, $y$ be two non-trivial elements in $G$. Then $\left\langle y \right\rangle \subset \left\langle x\right\rangle $ if and only if $\left\langle y^\pi\right\rangle \subset \left\langle x^\pi\right\rangle $. 
	\end{enumerate}
\end{lemma}

\begin{proof}
	(1) The result follows from the fact that $(x,y)\in A({\overrightarrow{\mathcal{RP}}(G)})$ if and only if $(x^\pi,y^\pi)\in A({\overrightarrow{\mathcal{RP}}(G)})$.
	
	(2) Let $x$ and $y$ be two non-trivial elements in $G$ such that $\left\langle y\right\rangle \subset \left\langle x\right\rangle $. We show that either $\left\langle y^\pi\right\rangle \subset \left\langle x^\pi\right\rangle $ or $\left\langle x^\pi\right\rangle \subset \left\langle y^\pi\right\rangle $. Notice that $\widehat{x}$ is either of Type~I or of Type~II and $\widehat{y}$ is either of Type~I or of Type IV. So we consider the following cases.
	
	\textit{Case} 1. Let $\widehat{x}$ be of Type~I. Then by Lemma~\ref{Auto lemma2}(1), $\mid \widehat{y}\mid >\mid \widehat{x}\mid $, if $\widehat{y}$ is of Type~IV with parameter $(2,p)$, where $p$ is a prime; $\mid \widehat{x}\mid \geq\mid \widehat{y}\mid $, otherwise. 
	
	\textit{Subcase} 1a. If $\widehat{y}$ is not of Type~IV with parameter $(2,p)$, then $\mid \widehat{x^\pi}\mid \geq\mid \widehat{y^\pi}\mid $, since $\mid \widehat{x}\mid \geq\mid \widehat{y}\mid $. Suppose $\left\langle x^\pi\right\rangle \subset \left\langle y^\pi\right\rangle $. Then $\widehat{y^\pi}$ is either of Type~I or of Type~II. By Lemma~\ref{Auto lemma2}, $\widehat{x^\pi}$ is of Type~IV with parameter $(2,q)$ and $\widehat{y^\pi}$ is of Type~ I with $o(y)=2q$, where $q$ is an odd prime. It follows that $\mid \widehat{y^\pi}\mid =(q-1)$ and $\mid \widehat{x^\pi}\mid =q$ and so $\mid \widehat{x}\mid $ is odd and $\mid \widehat{x}\mid >\mid \widehat{y}\mid $. Then by Lemma~\ref{relation of 2 eql class}, $\left\langle x\right\rangle \subset \left\langle y\right\rangle $, which is a contradiction. So $\left\langle y^\pi\right\rangle \subset \left\langle x^\pi\right\rangle $. 
	
	\textit{Subcase} 1b. If $\widehat{y}$ is of Type~IV with parameter $(2,p)$, then $\mid \widehat{y}\mid =p$. Since $\mid \widehat{y}\mid >\mid \widehat{x}\mid $, $\mid \widehat{y^\pi}\mid >\mid \widehat{x^\pi}\mid $ and $\mid \widehat{y^\pi}\mid $ is odd. By Lemma~\ref{relation of 2 eql class}, $\left\langle y^\pi\right\rangle \subset \left\langle x^\pi\right\rangle $.
	
	\textit{Case} 2. Suppose $\widehat{x}$ is of Type~II. By Lemma~\ref{Auto lemma2}(2), $\mid \widehat{x}\mid >\mid \widehat{y}\mid $. Then by a similar argument used in Subcase 1a, we get $\left\langle y^\pi\right\rangle \subset \left\langle x^\pi\right\rangle $.
	
	By a similar argument, we can show that if $\left\langle y^\pi\right\rangle \subset \left\langle x^\pi\right\rangle $, then $\left\langle y\right\rangle \subset \left\langle x\right\rangle $.
\end{proof}

\begin{pro}\label{order equal of eql class}
	Let $x\in G$ and $\pi\in Aut( \overrightarrow{\mathcal{RP}}(G))$. Then we have the following.
	\begin{enumerate}[\normalfont(1)]
		\item If $\widehat{x}$ is of Type~III, then $\widehat{x^\pi}=\widehat{x}$.
		
		\item If $\widehat{x}$ is of Type~II, then $\widehat{x^\pi}$ is  of Type~II with $o(x)=o(x^\pi)$.
		
		\item If $\widehat{x}$ is of Type~I, then $\widehat{x^\pi}$ is of Type~I with $o(x)=o(x^\pi)$.

		\item If $\widehat{x}$ is of Type~IV with parameter $(p,q)$, then $\widehat{x^\pi}$ is of Type~IV with parameter $(p,q)$.		
	\end{enumerate}
\end{pro}

\begin{proof}
	Notice that if $deg_{\overrightarrow{\mathcal{RP}}(G)}^+(x)>1$, then $deg_{\overrightarrow{\mathcal{RP}}(G)}^+(x^\pi)>1$. So $o(x)$ and $o(x^\pi)$ are composite numbers. By~\cite[Theorem 3.1]{raj RPG and PG 2018}, $o(x)$ and $o(x^\pi)$ can be determined from the structure of $\overrightarrow{\mathcal{RP}}(G)$. By Lemma~\ref{Auto presetve suset}(1), if $\left\langle y\right\rangle \subset \left\langle x\right\rangle $ (resp. $\left\langle x\right\rangle \subset \left\langle y\right\rangle $), then $\left\langle y^\pi\right\rangle \subset \left\langle x^\pi\right\rangle $ (resp. $\left\langle x^\pi\right\rangle \subset \left\langle y^\pi\right\rangle $). It follows that $o(x)=o(x^\pi)$.

	(1) If $\widehat{x}$ is of Type~III, then $N^-_{{\overrightarrow{\mathcal{RP}}}(G)}(x)=\emptyset$ and  $N^+_{{\overrightarrow{\mathcal{RP}}}(G)}(x)=\{e\}$. It follows that $N^-_{{\overrightarrow{\mathcal{RP}}}(G)}(x^\pi)=\emptyset$ and  $N^+_{{\overrightarrow{\mathcal{RP}}}(G)}(x^\pi)=\{e^\pi\}=\{e\}$, by Lemma~\ref{image of e}(2). Thus, $\widehat{x^\pi}$ is of Type~III. Then by Observation~\ref{ob2}(2), $\widehat{x^\pi}=\widehat{x}$. 
	
	(2) Assume that $\widehat{x}$ is of Type~II. If $deg^+_{{\overrightarrow{\mathcal{RP}}(G)}}(x)>1$, then $o(x)=o(x^\pi)$ and $deg^-_{{\overrightarrow{\mathcal{RP}}(G)}}(x^\pi)=0$. So $\widehat{x^\pi}$ is either of Type~I or of Type~II. By Observation~\ref{ob1}(2), $\mid \widehat{x}\mid =\mid \widehat{x^\pi}\mid =m\cdot o(x)$, where $m\geq 2$. It follows that $\widehat{x^\pi}$ is of Type~II. If $deg^+_{{\overrightarrow{\mathcal{RP}}(G)}}(x)=1$, then $N^+_{{\overrightarrow{\mathcal{RP}}(G)}}(x)=N^+_{{\overrightarrow{\mathcal{RP}}(G)}}(x^\pi)=\{e\}$. Also, $deg^-_{{\overrightarrow{\mathcal{RP}}(G)}}(x)=0$ and so $deg^-_{{\overrightarrow{\mathcal{RP}}(G)}}(x^\pi)=0$. It follows that $\widehat{x^\pi}=\widehat{x}$. Thus, $\widehat{x^\pi}$ is of Type~II and $o(x)=o(x^\pi)$.	
	
	(3) Assume that $\widehat{x}$ is of Type~I. If $deg^+_{{\overrightarrow{\mathcal{RP}}(G)}}(x)>1$, then $o(x)=o(x^\pi)$. Applying~\eqref{action3}, we get $\mid \widehat{x^\pi}\mid =\varphi(o(x^\pi))$, so $\widehat{x^\pi}$ is of Type~I. If $deg^+_{{\overrightarrow{\mathcal{RP}}(G)}}(x)=1$, then $deg^+_{{\overrightarrow{\mathcal{RP}}(G)}}(x^\pi)=1$. It follows that $o(x)$ and $o(x^\pi)$ are primes and so $\widehat{x^\pi}$ is one of Type~I, Type~III or Type~IV. Let $o(x)=p$ and $o(x^\pi)=q$, where $p$, $q$ are primes. If $\widehat{x^\pi}$ is of Type~II, then by part (2), $\widehat{x^\pi}=\widehat{x}$. Thus $\widehat{x}$ is also of Type~II, which is a contradiction. If $\widehat{x^\pi}$ is of Type~III, then by part (1), $\widehat{x}$ is also of Type~III, which is a contradiction. If $\widehat{x^\pi}$ is of Type~IV with parameter $(q,s)$, where $(q\neq)$ $s$ is a prime. Then there exists $w\in G$ such that $o(w)=qs$ and $N^+_{{\overrightarrow{\mathcal{RP}}(G)}}(w)=\widehat{x^\pi}\cup \{e\}$. It follows that there exists $y\in G$ such that $y^\pi=w$ and $N^+_{{\overrightarrow{\mathcal{RP}}(G)}}(y)=\widehat{x}\cup \{e\}$, so $o(y)=p^2$. Since $deg^+_{{\overrightarrow{\mathcal{RP}}(G)}}(y)>1$, by the previous argument, we have $o(y)=o(w)$, so $p^2=qr$, which is a contradiction. Thus $\widehat{x^\pi}$ is of Type~I, so $\mid \widehat{x^\pi}\mid =q-1$. Applying~\eqref{action3}, we have $o(x)=o(x^\pi)$.
	
	(4) Let $\widehat{x}$ be of Type~IV with parameter $(p,q)$, where $p$ and $q$ are distinct primes. Then $\widehat{x}=p+q-2$ and $deg^+_{{\overrightarrow{\mathcal{RP}}(G)}}(x^\pi)\neq 0$, $deg^-_{{\overrightarrow{\mathcal{RP}}(G)}}(x^\pi)=1$. This implies that $o(x^\pi)$ is prime and so $\widehat{x^\pi}$ is either of Type~I or of Type~IV. Suppose $\widehat{x^\pi}$ is of Type~I, then by part (3), $\widehat{x^\pi}^{\pi^{-1}}=\widehat{x}$ is of Type~I, which is a contradiction. Thus $\widehat{x^\pi}$ is of Type~IV. By~\eqref{action4}, $\mid \widehat{x^\pi}\mid =p+q-2$, so $\widehat{x^\pi}$ is of Type~IV with parameter $(p,q)$.
\end{proof}

\section{Full structure of $Aut\left(\mathcal{RP}(G)\right)$ and $Aut(\overrightarrow{\mathcal{RP}}(G))$}
In this section, we determine $Aut(\overrightarrow{\mathcal{RP}}(G))$ and
$Aut\left(\mathcal{RP}(G)\right)$.
\begin{lemma}\label{main}
	Let $\pi$ be a permutation on $G$.
	\begin{enumerate}[\normalfont(1)]
		\item If $\pi\in \mathscr{M}(G)$, then $ \left\langle x \right\rangle^\pi  =\left\langle x^\pi \right\rangle$ for each $x\in G$.\label{123}

		\item Then $\pi\in {\prod_{i=1}^{t}S_{\widehat{u_i}}}$ if and only if $ \widehat{x}^\pi =\widehat{x}$ for each $x\in G$.
		
		\item Suppose for every $x,y\in G$, $\pi$ is such that 
		\begin{enumerate}
			\item $o(x)=o(x^\pi)$;
			\item if $\left\langle x\right\rangle $ is a maximal subgroup of $G$ and $x\neq x^\pi$, then the set of  all subgroups of $\left\langle x\right\rangle $ and the the set of  all subgroups of $\left\langle x^\pi\right\rangle $ are different;
			\item $\left\langle y\right\rangle \subset \left\langle x\right\rangle $ if and only if $\left\langle y^\pi\right\rangle  \subset \left\langle x^\pi\right\rangle$ for every $x, y\in G$,
		\end{enumerate} then $\pi\in \mathscr{M}(G)$. 
	\end{enumerate}
	
\end{lemma}

\begin{proof} 
Parts (1) and (3) follow from \eqref{cyclic group permutation} and
	part~(2) follows from \eqref{action2}.		
\end{proof}

\begin{lemma}\label{subgroup of Aut}
	
	\begin{enumerate}[\normalfont(1)]
		\item $\mathscr{M}(G)$ is a subgroup of $Aut( \overrightarrow{\mathcal{RP}}(G))$ and $Aut({\mathcal{RP}}(G))$;
		
		\item ${\prod_{i=1}^{t}S_{\widehat{u_i}}}$ is a subgroup of $Aut( \overrightarrow{\mathcal{RP}}(G))$ and $Aut({\mathcal{RP}}(G))$. 
	\end{enumerate}
\end{lemma}

\begin{proof}	 
	(1) Let $\sigma\in \mathscr{M}(G)$. To show $\sigma\in Aut( \overrightarrow{\mathcal{RP}}(G))$, by  \eqref{cyclic group permutation}, it is enough to show that $(x,y)\in A({\overrightarrow{\mathcal{RP}}(G)})$ if and only if $(x^\sigma,y^\sigma)\in A({\overrightarrow{\mathcal{RP}}(G)})$.
	Suppose that $(x,y)\in A({\overrightarrow{\mathcal{RP}}(G)})$. Then $\left\langle y\right\rangle \subset 
	\left\langle x\right\rangle$. It follows from Lemma~\ref{main}(1) that $\left\langle y^\sigma \right\rangle \subset 
	\left\langle x^\sigma \right\rangle$. Therefore, $(x^\sigma,y^\sigma)\in A({\overrightarrow{\mathcal{RP}}(G)})$.
	Hence $P(G)$ is a subgroup of $Aut( \overrightarrow{\mathcal{RP}}(G))$. Since $Aut( \overrightarrow{\mathcal{RP}}(G))\subseteq Aut(\mathcal{RP}(G))$, it follows that $\mathscr{M}(G)$ is a subgroup of $Aut\left(\mathcal{RP}(G)\right)$.
	
	(2) Let $\tau\in {\prod_{i=1}^{t}S_{\widehat{u_i}}}$ and $(x,y)\in A(\overrightarrow{\mathcal{RP}}(G))$. Then $\left\langle y\right\rangle \subset \left\langle x \right\rangle$. By Lemma~\ref{main}(2), $y^\tau\in \widehat{y}$ and $x^\tau\in \widehat{x}$, so $\left\langle y^\tau\right\rangle \subset \left\langle x^\tau\right\rangle $. Hence $(x^\tau,y^\tau)\in A(\overrightarrow{\mathcal{RP}}(G))$. Conversely, assume that $(x^\tau,y^\tau)\in A(\overrightarrow{\mathcal{RP}}(G))$. Then $\left\langle y^\tau\right\rangle \subset \left\langle x^\tau\right\rangle $. This implies that $\left\langle y\right\rangle \subset \left\langle x\right\rangle $, since $x\in \widehat{x^\tau}$ and $y\in \widehat{y^\tau}$. Hence $(x,y)\in A(\overrightarrow{\mathcal{RP}}(G))$. Thus $\tau\in Aut( \overrightarrow{\mathcal{RP}}(G))$. It follows that ${\prod_{i=1}^{t}S_{\widehat{u_i}}}$ is a subgroup of $Aut( \overrightarrow{\mathcal{RP}}(G))$. Also, $Aut( \overrightarrow{\mathcal{RP}}(G))\subseteq Aut(\mathcal{RP}(G))$. This implies that ${\prod_{i=1}^{t}S_{\widehat{u_i}}}$ is a subgroup of $Aut(\mathcal{RP}(G))$.	
\end{proof}

\begin{lemma}\label{normal}
	$\mathscr{M}(G)$ is a subgroup of the normalizer of ${\prod_{i=1}^{t}S_{\widehat{u_i}}}$ in $Aut( \overrightarrow{\mathcal{RP}}(G))$ and $Aut(\mathcal{RP}(G))$.
\end{lemma}

\begin{proof}	
	Let $\pi\in \mathscr{M}(G)$, $\tau\in {\prod_{i=1}^{t}S_{\widehat{u_i}}}$ and $x\in G$. Then by Lemma~\ref{main}(2) and~\eqref{action4}, $\widehat{x}^{\pi\tau\pi^{-1}} =\widehat{x^\pi}^{\tau\pi^{-1}}=\widehat{x^\pi}^{\pi^{-1}}=\widehat{x}$. It follows that $\pi\tau\pi^{-1}\in {\prod_{i=1}^{t}S_{\widehat{u_i}}}$. Hence the proof.
\end{proof}

\begin{lemma}\label{intersection}
	$\mathscr{M}(G)\cap {\prod_{i=1}^{t}S_{\widehat{u_i}}}=\{ \textbf{1}_G \}$.
\end{lemma}

\begin{proof}
	Let $\pi \in \mathscr{M}(G)\cap {\prod_{i=1}^{t}S_{\widehat{u_i}}}$ and $x\in G$. Since $x=x_i^j$, where $\left\langle x_i\right\rangle \in M(G)$. Then $x^\pi=(x_i^{\pi})^j$. Since $\pi\in {\prod_{i=1}^{t}S_{\widehat{u_i}}}$ and $x_i^\pi\in \widehat{x_i}$, it follows that the set of all subgroups of $\left\langle x_i\right\rangle $ and the set of all subgroups of $\left\langle x_i^\pi\right\rangle $ are same; also $\pi\in \mathscr{M}(G)$, and so $\pi$ fixes $x_i$. Hence $x^\pi=x_i^j=x$. 
\end{proof}
The following result is an immediate consequence of Lemmas~\ref{subgroup of Aut},~\ref{normal} and~\ref{intersection}.

\begin{cor}\label{sub of aut}
	$\left( \prod_{i=1}^{t}S_{\widehat{u_i}} \right) \rtimes \mathscr{M}(G)$ is a subgroup of $Aut( \overrightarrow{\mathcal{RP}}(G))$ and $Aut\left(\mathcal{RP}(G)\right)$.
\end{cor}

\begin{lemma}\label{ext crs}
	Let $\pi\in Aut( \overrightarrow{\mathcal{RP}}(G))$. Then there exists $\sigma\in \mathscr{M}(G)$ such that $N_{\overrightarrow{\mathcal{RP}}(G)}^-(x^\pi)=N_{\overrightarrow{\mathcal{RP}}(G)}^-(x^\sigma)$ and $N_{\overrightarrow{\mathcal{RP}}(G)}^+(x^\pi)=N_{\overrightarrow{\mathcal{RP}}(G)}^+(x^\sigma)$ for every $x\in G$.
\end{lemma}

\begin{proof}
	Let $\pi\in Aut( \overrightarrow{\mathcal{RP}}(G))$. We define a map $\sigma$ on $G$ as follows: 	
	$e^\sigma:=e$. Let $x\in G$ with $x\neq e$. 	
	We have the following cases:
	
	\textit{Case} 1. If $\widehat{x}$ is of Type~I, then $w^\sigma:=w^\pi$ for every $w\in \widehat{x}$. Then $o(w)=o(w^\pi)$. If $\left\langle x\right\rangle $ is a maximal subgroup of $G$, then there is no maximal cyclic subgroup of $G$ such that the set of all its subgroups is the same as the set of all subgroups of $\left\langle w\right\rangle $, since  $\widehat{x}$ is of Type~I. Also, $\widehat{x}^\sigma =\widehat{x}^\pi=\widehat{x^\pi}$.
	
	\textit{Case} 2. If $\widehat{x}$ is of Type~II, and $x^\pi\notin \widehat{x}$, then by Proposition~\ref{order equal of eql class}(2), $\widehat{x^\pi}$ is also of Type~II with $o(x)=o(x^\pi)$. It follows that if $\widehat{x}=\displaystyle\bigcup_{i=1}^{r}[x_i]$, where $x_i\in G$, then $\widehat{x^\pi}=\displaystyle\bigcup_{i=1}^{r}[y_i]$, where $y_i=x_j^\pi$ for some $j\in\{1,2,\ldots,r\}$ and $o(x_i)=o(y_j)=o(x)$ for $i,j=1,2,\ldots,r$. Then for every $w\in \widehat{x}$, we have $w=[x_i]_l$ for some $i$ and $l$. For such a $w$, we define $w^\sigma:= [x_i]_l^\sigma=[x_i^\pi]_l$. It follows that $ \widehat{x}^\sigma=\widehat{x^\pi}$ for every $w\in \widehat{x}$.  Since $\widehat{x}\neq \widehat{x^\sigma}$, the set of  all subgroups of $\left\langle w\right\rangle $ and the  set of  all subgroups of $\left\langle w^\pi\right\rangle $ are not the same for every $w\in \widehat{x}$.

	If $\widehat{x}$ is of Type~II, and $x^\pi\in \widehat{x}$, then we define $w^\sigma:=w$ for every $w\in \widehat{x}$. It follows that $\widehat{x}^\sigma =\widehat{x}=\widehat{x}^\pi=\widehat{x^\pi}$.
	
	\textit{Case} 3. If $\widehat{x}$ is of Type~III, then $w^\sigma:=w$ for every $w\in \widehat{x}$.  It follows that $o(w)=o(w^\pi)$ for every $w\in \widehat{x}$ and $ \widehat{x}^\sigma=\widehat{x}$. So by Proposition~\ref{order equal of eql class}(1), $ \widehat{x}^\sigma=\widehat{x}=\widehat{x^\pi}$.
	
	\textit{Case} 4.	If $\widehat{x}$ is of Type~IV with parameter $(p,q)$, then by Proposition~\ref{order equal of eql class}(4), $\widehat{x^\pi}$ is also of Type~IV with parameter $(p,q)$. Then for each $w\in \widehat{x}$, $w^\sigma:=y$, where $y\in \widehat{x^\pi}$ such that $o(y)=o(w)$. So, $ \widehat{x}^\sigma=\widehat{x^\pi}$. Since $\widehat{x}$ is of Type~IV, $\left\langle w\right\rangle $ is not a maximal cyclic subgroup.
	
	Summarizing the above argument, we get  $\widehat{x^\sigma}= \widehat{x}^\sigma =\widehat{x^\pi}$ for every $\widehat{x}\in \mathscr{R}(G)$.
	
	
	Since $\simeq$ is an equivalence relation on $G$ and $\pi$ is well defined, it follows that $\sigma$ is well defined. Let $A_1, A_{2}, A_3$ and $A_4$ be the set of all elements in the $\simeq$-classes of Type~I, II, III, IV, respectively. Then $G$ is the disjoint union of $\{e\}, A_1, A_2, A_3, A_4$. By the definition of $\sigma$, it is clear that $o(x)=o(x^\sigma)$ and the  map $\sigma\mid_{A_i}$ is a bijection on $A_i$ for $i=1,2,3,4$. Thus $\sigma$ is a bijection on $G$. 
	
	Now let $x,y\in G$. Then by Lemma~\ref{Auto presetve suset}(2), $\left\langle y\right\rangle \subset \left\langle x\right\rangle $ if and only if $\left\langle y^\pi\right\rangle \subset \left\langle x^\pi\right\rangle $. It follows that $\left\langle y\right\rangle \subset \left\langle x\right\rangle $ if and only if $\left\langle y^\sigma\right\rangle \subset \left\langle x^\sigma\right\rangle $, since $\widehat{x^\sigma}=\widehat{x^\pi}$. So by Lemma~\ref{main}(3), $\sigma\in P(G)$ and $N_{\overrightarrow{\mathcal{RP}}(G)}^-(x^\pi)=N_{\overrightarrow{\mathcal{RP}}(G)}^-(x^\sigma)$, $N_{\overrightarrow{\mathcal{RP}}(G)}^+(x^\pi)=N_{\overrightarrow{\mathcal{RP}}(G)}^+(x^\sigma)$ for every $x\in G$. 
\end{proof}

%
\begin{thm}\label{auto}
	Let $G$ be a finite group. Then
	
	$$Aut( \overrightarrow{\mathcal{RP}}(G))= \left(  {\prod_{i=1}^{t}S_{\widehat{u_i}}}\right)  \rtimes \mathscr{M}(G),$$
	
	\noindent where  $\mathscr{M}(G)$ and ${\prod_{i=1}^{t}S_{\widehat{u_i}}}$ act on $G$ as in~\eqref{cyclic group permutation} and~\eqref{action2}, respectively. 
\end{thm}

\begin{proof}
	Let $\pi\in Aut( \overrightarrow{\mathcal{RP}}(G))$. Then by Lemma~\ref{ext crs}, there exists $\sigma\in \mathscr{M}(G)$ such that $N_{\overrightarrow{\mathcal{RP}}(G)}^-(x^\pi)=N_{\overrightarrow{\mathcal{RP}}(G)}^-(x^\sigma)$ and $N_{\overrightarrow{\mathcal{RP}}(G)}^+(x^\pi)=N_{\overrightarrow{\mathcal{RP}}(G)}^+(x^\sigma)$ for every $x\in G$. Thus $$N_{\overrightarrow{\mathcal{RP}}(G)}^-(x^{\pi\sigma^{-1}}) =N_{\overrightarrow{\mathcal{RP}}(G)}^-((x^{\sigma^{-1}})^\pi)=N_{\overrightarrow{\mathcal{RP}}(G)}^-((x^{\sigma^{-1}})^\sigma)= N_{\overrightarrow{\mathcal{RP}}(G)}^-(x)$$ and similarly, $N_{\overrightarrow{\mathcal{RP}}(G)}^+(x^{\pi\sigma^{-1}})=N_{\overrightarrow{\mathcal{RP}}(G)}^+(x)$. Therefore, $\widehat{x}^{\pi\sigma^{-1}}=\widehat{x}$. So by Lemma~\ref{main}(2), $\pi\sigma^{-1}\in {\prod_{i=1}^{t}S_{\widehat{u_i}}}$. Hence $\pi\in \left(  {\prod_{i=1}^{t}S_{\widehat{u_i}}}\right)  \mathscr{M}(G)$ and so the result follows from Corollary~\ref{sub of aut}. 
\end{proof}

\begin{thm}\label{autRP}
	\begin{enumerate}[\normalfont(1)]
		\item Let $G$ be a finite group. If $G\ncong \mathbb Z_{2^m}(m\geq 1)$ and $Q_{2^\alpha}(\alpha\geq 3)$, then $$Aut\left(\mathcal{RP}(G)\right)=\left( {\prod_{i=1}^{t}S_{\widehat{u_i}}}\right)  \rtimes \mathscr{M}(G),$$ 
		\noindent where  $\mathscr{M}(G)$ and ${\prod_{i=1}^{t}S_{\widehat{u_i}}}$ act on $G$ as in~\eqref{cyclic group permutation} and~\eqref{action2}, respectively.
		
		\item $ Aut(\mathcal{RP}(\mathbb Z_{2^m}))=\mathbb Z_2\times {\displaystyle \prod_{i=2}^{m}}S_{\varphi(2^i)} $ for $m\geq 2$.
		
		\item $ Aut(\mathcal{RP}(Q_{2^\alpha}))=\mathbb Z_2\times \left( {\displaystyle\prod_{i=2}^{\alpha-1}}S_{\varphi(2^i)}\right) \times S_{2^{\alpha-1}}$ for $\alpha\geq 3$.
	\end{enumerate}
\end{thm}
\begin{proof}
	(1)  Let $\pi\in Aut\left(\mathcal{RP}(G)\right)$. Let $x, y\in G$. Suppose $x$ and $y$ are non-trivial elements in $G$. Then by Lemma~\ref{Auto presetve suset}(2), $\left\langle y\right\rangle \subset \left\langle x\right\rangle $ if and only if $\left\langle y^\pi\right\rangle \subset \left\langle x^\pi\right\rangle $. Suppose either $x$ or $y$ is the trivial element. Without loss of generality, we assume that $y=e$. Then by Lemma~~\ref{image of e}(1), $e^\pi=e$ and so $\left\langle y\right\rangle \subset \left\langle x\right\rangle $ and $\left\langle y^\pi\right\rangle \subset \left\langle x^\pi\right\rangle $.  It follows that $(x,y)\in {\overrightarrow{\mathcal{RP}}(G)}$ if and only if $(x^\pi,y^\pi)\in {\overrightarrow{\mathcal{RP}}(G)}$. Thus $\pi\in Aut( \overrightarrow{\mathcal{RP}}(G))$ and so $Aut(\mathcal{RP}(G))\subseteq Aut( \overrightarrow{\mathcal{RP}}(G))$. Thus $Aut(\mathcal{RP}(G))= Aut( \overrightarrow{\mathcal{RP}}(G))$, since $Aut( \overrightarrow{\mathcal{RP}}(G))\subseteq  Aut(\mathcal{RP}(G))$. 	
	So the proof follows from Theorem~\ref{auto}.
	
	(2)-(3) Let $\pi\in Aut\left(\mathcal{RP}(G)\right)$. By Lemma~~\ref{image of e}(1), $e^\pi=e$ or $a$, where $a$ is the unique element of order $2$ in $G$. Let $x,y\in G$ be such that $x,y\neq e,a$. Then $deg_{\mathcal{RP}(G)}(x)=deg_{\mathcal{RP}(G)}(y)$ if and only if $\widehat{x}=\widehat{y}$.
	It follows that $\widehat{x}^\pi=\widehat{x}$ for every $x\neq e,a$ and so $\pi\mid_{\{G\setminus\{e,a\}\}}\in{\prod_{i=1}^{t}S_{\widehat{u_i}}}$. So we get \begin{align}\label{thm2eqn} Aut\left(\mathcal{RP}(G)\right)=\mathbb Z_2\times {\prod_{i=1}^{t}S_{\widehat{u_i}}}.
	\end{align}

	It is not hard to see that  $\mathscr{R}(\mathbb Z_{2^m})=\left\lbrace \widehat{u_1}, \widehat{u_2}, \ldots , \widehat{u_t}\right\rbrace $, where $t=m+1$ and $\widehat{u_i}$ is the set of all generators of the unique cyclic subgroup of order $2^{i-1}$ in $\mathbb Z_{2^m}$ for $i=1,2,\ldots ,t$. Let $a$ be an element of order $2^{\alpha -1}$ in $Q_{2^\alpha}$. Then $\mathscr{R}(Q_{2^\alpha})=\left\lbrace \widehat{u_1}, \widehat{u_2}, \ldots ,  \widehat{u_{t}}\right\rbrace $, where $t=\alpha+1$, $\widehat{u_i}$ is the set of all generators of the unique cyclic subgroup of order $2^{i-1}$ in $\langle a \rangle$ for  $i=1,2,\ldots t-1$ and $\widehat{u_{t}}$ is the set of all elements in $Q_{2^\alpha}\setminus\langle a \rangle$. Substituting these in~\eqref{thm2eqn}, we get the results.
\end{proof}

\section{Examples}
In this section, we compute the automorphism groups of the reduced power (di)graphs of several classes of finite groups. 
\begin{example}\label{Aut(Zn)}
	Let $n\geq 1$ be an integer and let $p$, $q$ be distinct primes. Then
	\begin{enumerate}[\normalfont(1)]
		\item $Aut(\overrightarrow{\mathcal{RP}}(\mathbb Z_n))) \cong \begin{cases}
			{\prod_{d\mid n}}S_{\varphi(d)}, & \text{if}~ n\neq pq;\\
			S_{\varphi(pq)}\times S_{p+q-2}, & \text{if}~n=pq.
		\end{cases}$
		\item $Aut(\mathcal{RP}(\mathbb Z_n))\cong \begin{cases}
			{\prod_{d\mid n}}S_{\varphi(d)}, & ~\text{if}~ n\neq pq, 2^\alpha;\\
			S_{\varphi(pq)}\times S_{p+q-2}, & ~\text{if}~n=pq;\\
			S_2\times  \prod_{i=2}^{\alpha}S_{\varphi(2^i)}, &~\text{if}~n=2^\alpha,
		\end{cases}$\\
		where $\alpha\geq 1$.
	\end{enumerate} 
\end{example}

\begin{proof}
	Notice that $\mathscr{M}(\mathbb Z_n)$ is a singleton, since $\mathbb Z_n$ has exactly one maximal cyclic subgroup. For each $d\in D(n)$,  let $A_d$ denote the unique cyclic subgroup of order $d$ in  $\mathbb Z_n$, where $D(n)$ denote the set of all positive divisors of $n$. Then $S_{[A_d]}\cong S_{\varphi(d)}$.
	Notice that \begin{align}\label{R(Zn)}
		\mathscr{R}(\mathbb Z_n) =\begin{cases}
			\{ [A_d]\mid d\in D(n)\} & \text{if~} n\neq pq;\\
			\{ [A_d]\mid d=1,n\}\cup \{[A_p]\cup [A_q]\}& \text{if~} n= pq;
		\end{cases}
	\end{align}		 
	Applying these in Theorems~\ref{auto} and \ref{autRP}, we get the result.
\end{proof}

\begin{example}
	Let $n\geq 2$ be an integer and let $p$, $q$ be distinct primes. Then
	\[ Aut( \overrightarrow{\mathcal{RP}}(D_{2n})) =Aut({\mathcal{RP}}(D_{2n}))\cong \begin{cases}
		\left({ \prod_{d\mid n}}S_{\phi(d)}\right)\times S_n, &~\text{if}~n\neq pq;\\
		S_{\varphi(pq)}\times S_{p+q-2}\times S_{pq}, &~\text{if}~n= pq. 
	\end{cases}\]	
\end{example}

\begin{proof}
Notice that	$D_{2n}=\left\langle a,b \mid a^n=e=b^2,\right.$ $\left. ab=ba^{-1} \right\rangle=\{e,a,a^2,\ldots,a^n,b,ab,\ldots ,a^{n-1}b\}$. Here $M(D_{2n})=\{\left\langle a\right\rangle \}\cup \{\langle a^ib\rangle \mid 0\leq i\leq n-1\}$. Also, $o(a^ib)=2, \mid\langle a^ib\rangle \cap \langle a\rangle\mid=1$ and $\mid\langle a^ib\rangle \cap \langle a^jb\rangle\mid=1$ for $i\neq j$. Hence $\mathscr{M}(D_{2n})=\{\textbf{1}_{M(D_{2n})}\}$.
	Also, $\mathscr{R}(D_{2n})=\mathscr{R}(\left\langle a\right\rangle)\cup  \widehat{b}$, where $\widehat{b}=\{a^ib\mid 1\leq i\leq n\}$ (cf. \cite[Figure 2]{raj RPG 2018}). 		
	Applying these in Theorems~\ref{auto} and~\ref{autRP}, we get
	\begin{align*}
		Aut( \overrightarrow{\mathcal{RP}}(D_{2n}))=Aut( \mathcal{RP}(D_{2n}))&= \left( {\prod_{\widehat{u}\in \mathscr{R}(\left\langle a\right\rangle) }S_{\widehat{u}}}\right) \times S_{\widehat{b}} \\
		&= \left( {\prod_{\widehat{u}\in \mathscr{R}(\mathbb Z_n) }S_{\widehat{u}}}\right) \times S_{\widehat{b}}.
	\end{align*}
	So the result follows from~\eqref{R(Zn)}.
\end{proof}

\begin{example}\label{AutQm}
	Let $m\geq 2$ be an integer and let $p$ be an odd prime. Then
	\begin{enumerate}[\normalfont(1)]
		\item $Aut( \overrightarrow{\mathcal{RP}}(Q_{4m})) \cong \begin{cases}
			\left(	{\prod_{d\mid 2m}}S_{\varphi(d)}\right)\times S_{2m}, & ~\text{if}~ m\neq p;\\
			S_{p-1}\times S_{p}\times S_{2p}, & ~\text{if}~m=p.
		\end{cases}$
		\item $Aut(\mathcal{RP}(Q_{4m}))\cong \begin{cases}
			\left({\prod_{d\mid 2m}}S_{\varphi(d)}\right)\times S_{2m}, & ~\text{if}~ m\neq p, 2^\alpha;\\
			S_{p-1}\times S_{p}\times S_{2p}, & ~\text{if}~m=p;\\
			S_2\times \left(\prod_{i=2}^{\alpha+1}S_{\varphi(2^i)}\right)\times S_{2^{\alpha+1}}, &~\text{if}~m=2^\alpha,
		\end{cases}$\\
		where $\alpha\geq 1$.
	\end{enumerate} 
\end{example}

\begin{proof}
Notice that	$Q_{4m}=\left\langle a, b\mid a^{2m}=e=b^4, bab^{-1}=a^{-1} \right\rangle=\{e,a,a^2,\ldots,a^{2m-1},b,ab,a^2b,\\\ldots,a^{2m-1}b\}$. Here $M(Q_{4m})=\{ \left\langle a\right\rangle \}\cup\langle a^ib \rangle  \mid 0\leq i\leq m-1 \rbrace  $. Clearly, $o(a)=2m$. It can be seen that for every $i=0,1,\ldots m-1$, $o(a^ib)=4$ and  $\left\langle a^ib\right\rangle $  has the subgroup of order $2$ as common. Hence $\mathscr{M}(Q_{4m})=\{\textbf{1}_{M(Q_{4m})}\}$.
	Also $\mathscr{R}(Q_{4m})=\mathscr{R}(\left\langle a\right\rangle)\cup  \widehat{b}$, where $\widehat{b}=\{a^ib\mid 1\leq i\leq 2m\}$ (cf. \cite[Figure 3]{raj RPG 2018}). 		
	Hence by Theorem~\ref{auto},
	\begin{align*}
		Aut( \overrightarrow{\mathcal{RP}}(Q_{4m}))&=  {\left(\prod_{\widehat{u}\in \mathscr{R}(\left\langle a\right\rangle) }S_{\widehat{u}}\right)}  \times S_{\widehat{b}} \\
		&=\left({\prod_{\widehat{u}\in \mathscr{R}(\mathbb Z_{2m}) }S_{\widehat{u}}}\right) \times S_{\widehat{b}}.
	\end{align*}
	So the part~(1) of this result follows from~\eqref{R(Zn)}. 
	
	In view of Theorems~\ref{auto} and~\ref{autRP}, it follows that $Aut( \overrightarrow{\mathcal{RP}}(Q_{4m}))=Aut( \mathcal{RP}(Q_{4m}))$ for $m\neq 2^\alpha$, where $\alpha \geq 2$. So the part~(2) of this result follows from Theorem~\ref{autRP}(3). 
\end{proof}

\begin{example}
	Let $n\geq 2$ be an integer. Then
	$$Aut( \overrightarrow{\mathcal{RP}}(SD_{8n})) =Aut({\mathcal{RP}}(SD_{8n}))\cong  S_{2n}\times S_{2n}\times
	{\prod_{d\mid 4n}}S_{\phi(d)} .$$	 
\end{example}

\begin{proof}
	$SD_{8n}=\left\langle a, b\mid a^{4n}=e =b^2, bab^{-1}=\right.$ $\left.a^{2n-1} \right\rangle =\{e,a,a^2\ldots, a^{4n-1}\}\cup \{a^ib\mid 1\leq i\leq 4n~\text{and $i$ is even}\} \cup \{a^ib\mid 1\leq i\leq 4n~\text{and $i$ is odd}\}$. Here $M(SD_{8n})=\{\langle a \rangle \}\cup \{\langle a^ib\rangle \mid 1\leq i\leq 4n~\text{and $i$ is even}\} \cup \{\langle a^ib\rangle \mid 1\leq i\leq 2n~\text{and $i$ is odd}\}$. Notice that $o(a^ib)=2$, where $1\leq i\leq 4n$ and $i$ is even; $o(a^ib)=4$, where $1\leq i\leq 2n$ and $i$ is odd;  $\langle a^ib\rangle \cap \langle a\rangle=\langle a^{2n}\rangle $, where $1\leq i\leq 2n$ and $i$ is odd  and $\langle a^ib\rangle \cap \langle a^jb\rangle=\langle a^{2n}\rangle $,  where $1\leq i,j\leq 2n$, $i\neq j$ and $i$ is odd. Hence $\mathscr{M}(SD_{8n})=\{\textbf{1}_{M(SD_{8n})}\}$.
	Also, $\mathscr{R}(SD_{8n})=\mathscr{R}(\left\langle a\right\rangle)\cup  \widehat{b}\cup \widehat{ab}$, where $\widehat{b}=\{a^ib\mid 1\leq i\leq 4n~\text{and $i$ is even}\}$ and $\widehat{ab}=\{a^ib\mid 1\leq i\leq 4n~\text{and $i$ is odd}\}$ (cf. \cite[Figure 4]{raj RPG 2018}).
	Applying these in Theorems~\ref{auto} and~\ref{autRP},
	\begin{align*}
		Aut( \overrightarrow{\mathcal{RP}}(SD_{8n}))=Aut( \mathcal{RP}(SD_{8n}))&= \left( {\prod_{\widehat{u}\in \mathscr{R}(\left\langle a\right\rangle) }S_{\widehat{u}}}\right) \times S_{\widehat{b}}\times S_{\widehat{ab}} \\
		&= \left({\prod_{\widehat{u}\in \mathscr{R}(\mathbb Z_{4n}) }S_{\widehat{u}}}\right) \times S_{\widehat{b}}\times S_{\widehat{ab}}.
	\end{align*}
	So the result follows from~\eqref{R(Zn)}.
\end{proof}

\begin{example}		
	Let $G$ be a finite group of order $n$ which is isomorphic to a $p$-group with exponent $p$ (except $\mathbb Z_2$) or a non-nilpotent group of order $n$$(=p^mq)$ with all non-trivial elements are of order $p$ or $q$, where $p$, $q$ are distinct primes. Then $Aut( \overrightarrow{\mathcal{RP}}(G)) =Aut( {\mathcal{RP}}(G)) \cong S_{n-1}$.		
\end{example}

\begin{proof}
	Let $G$ be any one of the groups mentioned in this example. According to~\cite{prime elt}, every non-trivial element of $G$ is of order either $p$ or $q$. Hence all the proper subgroups of $G$ are maximal cyclic subgroups of prime order and the set of all proper subgroups of each of these maximal subgroups are the same, which is clearly $\{e\}$.  Hence $\mathscr{M}(G)=\{\textbf{1}_{M(G)}\}$. Also, $\mathscr{R}(G)=\{ \{e\}, G\setminus \{e\}\}$. Applying these in Theorems~\ref{auto} and~\ref{autRP}, we get
	
	$$Aut(\overrightarrow{\mathcal{RP}}(G))=Aut(\mathcal{RP}(G))=  S_{G\setminus \{e\}}
	=S_{n-1}.$$
	
	\noindent	 	So the result follows.	
\end{proof}

\begin{example}
	Let $n\geq 1$ be an integer and $p$ be a prime. Then
	
	 $Aut( \overrightarrow{\mathcal{RP}}(\mathbb Z_{p^2}^n))=Aut\left(\mathcal{RP}(\mathbb Z_{p^2}^n)\right)\cong \left(S_{p-1}\times S_{p^n(p-1)}\right)\wr S_m$, 	 
	 where $m=\frac{p^n-1}{p-1}$.
\end{example}

\begin{proof}
	Notice that $\mathbb Z_{p^2}^n$ has $m=\frac{p^n-1}{p-1}$ subgroups of order $p$. Let them be $\left\langle x_i\right\rangle $ for $i=1,2,\ldots, m$. Also, each $\left\langle x_i\right\rangle $ is contained in exactly $p^{n-1}$ cyclic subgroups of order $p^2$. So $M(G)$ is the set of all cyclic subgroups of order $p^2$. It follows that $\mathscr{M}(\mathbb Z_{p^2}^n)\cong S_{m}$. Also, $\mathscr{R}(\mathbb Z_{p^2}^n)=\{e,\widehat{x_1},\widehat{x_2},\ldots , \widehat{x_m},\widehat{b_1},\widehat{b_2},\ldots, \widehat{b_{m}}\}$, where $\left\langle b_i\right\rangle $ is the cyclic subgroup of order $p^2$ containing $\left\langle x\right\rangle $. Hence $\mid \widehat{x_i}\mid =p-1$ and $\mid \widehat{b_i}\mid=p^n(p-1)$. Applying these in Theorems~\ref{auto} and~\ref{autRP}, we get
	\begin{align*}
		Aut( \overrightarrow{\mathcal{RP}}(\mathbb Z_{p^2}^n))=Aut( \mathcal{RP}(\mathbb Z_{p^2}^n))&=  \left[{\prod_{i=1}^{m}}\left(S_{p-1} \times S_{p^n(p-1)}\right)  \right] \ltimes S_m\\
		&= \left(S_{p-1} \times S_{p^n(p-1)}\right)\wr S_m.
	\end{align*}
This completes the proof.
\end{proof}

\begin{example}\label{modular}
	Let $\alpha\geq 3$ be an integer and $p$ be a prime. Then we have the following.
	\begin{enumerate}[\normalfont(1)]
		\item $Aut( \overrightarrow{\mathcal{RP}}(\mathbb M_{8}))=Aut( \mathcal{RP}(\mathbb M_{8}))=S_2\times S_4$.
		
		\item $Aut( \overrightarrow{\mathcal{RP}}(\mathbb M_{p^\alpha}))=Aut( \mathcal{RP}(\mathbb M_{p^\alpha}))=S_{p^{\alpha-1}(p-1)}\times S_{p(p-1)}\times G_1  \times G_2$,
	\end{enumerate}
	where $p^\alpha \neq 8$ and $G_j=\displaystyle{\prod_{i=1}^{\alpha-j-1}}S_{p^{\alpha-i-2}(p-1)^{j}}$for $j=1,2$.
\end{example}

\begin{proof}
	(1) From the subgroup lattice of $\mathbb M_8$ (cf.~\cite[Figure 3]{bohanan}), it is clear that $\mathbb M_8$ contains four maximal cyclic subgroups of order 2 and a unique maximal cyclic subgroups of order 4. It follows that $\mathscr{M}(\mathbb M_8)=I_{M(\mathbb M_8)}$. Also, $\mathscr{R}(\mathbb M_8)=\mathscr{R}(\mathbb Z_4)\cup \{\widehat{b}\}$, where $\widehat{b}=\{b,a^2b,ab,a^3b\}$. Applying these in Theorems~\ref{auto} and~\ref{autRP}, we get the result~(1).
	
	(2) From the subgroup lattice of $\mathbb M_{p^\alpha}$, where $p^\alpha \neq 8$ (cf.~\cite[Figure 4]{bohanan}), it is clear that $\mathbb M_{p^\alpha}$ has $p$ number of cyclic subgroups of order $p^i$ for $i=2,3,\ldots, p^n$ and $p+1$ cyclic subgroups of order $p$. Hence $M(G)$ is the set of all cyclic subgroups of order $p^n$.   Moreover, all the maximal cyclic subgroups have the same set of proper cyclic subgroups. So $\mathscr{M}(G)=\{\textbf{1}_{M(G)}\}$. Also, 	
	$$\mathscr R(\mathbb M_{p^\alpha})=\{\widehat{a}, \widehat{b}, e\}\displaystyle\bigcup_{i=1}^{\alpha-2}\left\lbrace \widehat{a^{p^i}}\right\rbrace  \displaystyle\bigcup_{i=1}^{\alpha-3}\left\lbrace\widehat{a^{p^i}b}\right\rbrace, $$	
	 where $o(b)=p$, $o(a^{p^i})=p^{\alpha-{(i+1)}}=o(a^{p^i}b)$, $\widehat{a}=\displaystyle\bigcup_{j=1}^{p-1} [ab^j] \cup [a]$, $\widehat{b}=\displaystyle\bigcup_{j=1}^{p-1} [a^{\alpha-2}b^j] \cup [b]$, $\widehat{a^{p^i}}=[a^{p^i}]$, and $\widehat{a^{p^i}b}=\displaystyle \bigcup_{j=1}^{p-1}[a^{p^i}b^j]$ for 	
	$i=1,2,\ldots \alpha-2$. 
	
	Applying these in Theorems~\ref{auto} and~\ref{autRP}, we get
	\begin{align*}
		Aut( \overrightarrow{\mathcal{RP}}(\mathbb M_{p^\alpha}))=Aut( \mathcal{RP}(\mathbb M_{p^\alpha}))&= S_{\widehat{a}}\times S_{\widehat{b}}\times \left( {\prod_{i=1}^{\alpha-2}}S_{\widehat{a^{p^i}}} \right)\times \left({\prod_{i=1}^{\alpha-3}} S_{\widehat{a^{p^i}b}}\right)\\
		&= S_{p^{\alpha-1}(p-1)}\times S_{p(p-1)}\times G_1  \times G_2,
	\end{align*} 
	where $G_j=\displaystyle{\prod_{i=1}^{\alpha-j-1}}S_{p^{\alpha-i-2}(p-1)^{j}}$for $j=1,2$.
\end{proof}

Note that the subgroup lattice of $\mathbb Z_{p^{\alpha-1}}\times \mathbb Z_p$ is isomorphic to the subgroup lattice of $\mathbb M_{p^\alpha}(p^\alpha\neq 8)$. As a consequence, we get the following result by using Example~\ref{modular}(2).
\begin{example}
	Let $\alpha\geq 2$ be an integer and $p$ be a prime. Then 
	
	$Aut( \overrightarrow{\mathcal{RP}}(\mathbb Z_{p^{\alpha-1}}\times \mathbb Z_p))=Aut( \mathcal{RP}(\mathbb Z_{p^{\alpha-1}}\times \mathbb Z_p))= S_{p^{\alpha-1}(p-1)}\times S_{p(p-1)}\times G_1  \times G_2$, 
	
	where $p^\alpha \neq 8$ and $G_j=\displaystyle{\prod_{i=1}^{\alpha-j-1}}S_{p^{\alpha-i-2}(p-1)^{j}}$ for $j=1,2$.
\end{example}

\begin{example}
	Let $n$ be an odd positive integer. Then
	
	$Aut( \overrightarrow{\mathcal{RP}}(V_{8n}))=Aut( \mathcal{RP}(V_{8n}))=S_{2n}\times S_{2n}\times \displaystyle \left(\prod_{d\mid 2n}S_{\varphi(d)}\right)\times \left(\displaystyle\prod_{d\mid 2n, d\nmid n}S_{\varphi(d)}\right)\wr S_2 $. 
\end{example}

\begin{proof}
	Notice that 
	\begin{center}
		$M(V_{8n})=\left\lbrace \left\langle ab^2\right\rangle ,\left\langle a\right\rangle ,\left\langle a^2b^2\right\rangle \right\rbrace \bigcup \left\lbrace \left\langle a^jb \right\rangle \mid  j~ \text{is even}  \right\rbrace\bigcup$ $   \left\lbrace \left\langle a^jb^k \right\rangle \mid  j~ \text{is odd},~k=1~ \text{or}~3  \right\rbrace$,
	\end{center} where $o(ab^2)=o(a)=o(a^2b^2)=2n$ and $o(a^jb)=4$, if $j$ is even; and  $o(a^jb^k)=2$, if $j$ is odd and $k=1$ or 3. Also, $\left\langle a^ib\right\rangle \cap \left\langle a^jb\right\rangle =\left\langle b^2\right\rangle $. It follows that each $\sigma\in \mathscr{M}(V_{8n})$ fixes all the elements in $\left\lbrace \left\langle a^jb \right\rangle \mid  j~ \text{is even}  \right\rbrace$ and $\left\lbrace \left\langle a^jb^k \right\rangle \mid  j~ \text{is odd},~k=1~ \text{or}~3  \right\rbrace$. Notice that $\left\langle a^2b^2\right\rangle $ contains a subgroup $\left\langle b^2\right\rangle $ of order 2 which is contained in $n$ number of subgroups of order 4; but $\left\langle a\right\rangle $ and $\left\langle a^2b\right\rangle $ does not contain such a subgroup. It turns out that $\mathscr{M}(V_{8n})\cong S_2$.  
	
	Suppose $\mathit{Z}_n$ denotes the reduced power graph of a cyclic group of order $n$. Then the structure of $\mathcal{RP}(V_{8n})$ is as shown in~\cite[Figure 3]{Ali}. From this structure, we have $$\mathscr{R}(V_{8n})= \widehat{a^2b}\cup \widehat{ab}\cup \mathscr{R}\left( \left\langle a^2b^2 \right\rangle \right) \cup \left[ \mathscr{R}\left( \left\langle ab^2 \right\rangle  \right)\setminus \mathscr{R}\left( \left\langle a^2\right\rangle \right)\right]  \cup \left[ \mathscr{R}\left( \left\langle a \right\rangle  \right)\setminus \mathscr{R}\left( \left\langle a^2\right\rangle \right)\right], $$ where $\widehat{a^2b}=\left\lbrace  a^jb^k \mid  j~ \text{is even},~k=1~ \text{or}~3  \right\rbrace$ and $\widehat{ab}=\left\lbrace  a^jb^k \mid  j~ \text{is odd},~k=1~ \text{or}~3  \right\rbrace$. 
	
	Applying these in Theorems~\ref{auto} and~\ref{autRP}, we get\\
	
	\noindent $Aut( \overrightarrow{\mathcal{RP}}(V_{8n}))=Aut( \mathcal{RP}(V_{8n}))$
	\begin{align*}
		&= S_{2n}\times S_{2n}\times  \left( {\prod_{\widehat{u}\in \mathscr{R}(\left\langle a^2b^2\right\rangle) }S_{\widehat{u}}}\right)  \times \left[\left( {\prod_{\widehat{u}\in \mathscr{R}(\left\langle ab^2\right\rangle), \widehat{u}\notin \mathscr{R}(\left\langle a^2\right\rangle) }S_{\widehat{u}}} \right) \times \left( {\prod_{\widehat{u}\in \mathscr{R}(\left\langle a\right\rangle), \widehat{u}\notin \mathscr{R}(\left\langle a^2\right\rangle)}S_{\widehat{u}}} \right) \right]\ltimes S_2\\
		&=S_{2n}\times S_{2n}\times \left( {\prod_{\widehat{u}\in \mathscr{R}(\mathbb Z_{2n}) }S_{\widehat{u}}} \right) \times \left[\left(  {\prod_{\widehat{u}\in \mathscr{R}(\mathbb Z_{2n}),\widehat{u}\notin \mathscr{R}(\mathbb Z_{n}) }S_{\widehat{u}}}\right)  \times \left(  {\prod_{\widehat{u}\in \mathscr{R}(\mathbb Z_{2n}),\widehat{u}\notin \mathscr{R}(\mathbb Z_{n}) }S_{\widehat{u}}}\right) \right] \ltimes S_2\\
		&=S_{2n}\times S_{2n}\times \left( \displaystyle\prod_{d\mid 2n}S_{\varphi(d)}\right) \times\left[\left( \displaystyle\prod_{d\mid 2n, d\nmid n}S_{\varphi(d)}\right) \times \left( \displaystyle\prod_{d\mid 2n, d\nmid n}S_{\varphi(d)}\right) \right]\ltimes S_2\\ 
		&=S_{2n}\times S_{2n}\times \left( \displaystyle\prod_{d\mid 2n}S_{\varphi(d)}\right) \times \left(  \displaystyle\prod_{d\mid 2n, d\nmid n}S_{\varphi(d)}\right)\wr S_2.
	\end{align*}
This completes the proof.
\end{proof}

\begin{example}\label{autV8n}
	\begin{enumerate}[\normalfont(1)]
		\item Let $n=2^kt$, where $t$ is an odd positive integer and $k$ is a positive integer with $n\neq 2$. Then
		$Aut( \overrightarrow{\mathcal{RP}}(V_{8n}))=Aut( \mathcal{RP}(V_{8n}))=S_{2n}\times S_{2n}\times S_{2n}\times \left( \displaystyle\prod_{d\mid 2^{k}t}S_{\varphi(d)}\right) \times \left( \displaystyle\prod_{d\mid 2t, d\nmid t}S_{\varphi(d)}\right) \times \left( \displaystyle\prod_{d\mid 2^{k}t, d\nmid t}S_{\varphi(d)}\right) $. 
		\item $Aut( \overrightarrow{\mathcal{RP}}(V_{16}))=Aut( \mathcal{RP}(V_{16}))=S_5 \times (S_{4}\wr S_{2})$.
	\end{enumerate} 
\end{example}

\begin{proof}
	(1) Assume that $n\neq 2$.  	It can be seen that
	\begin{align*}
		 M(V_{8n})=&\left\lbrace\left\langle a\right\rangle ,\left\langle a^{2^{k+1}}b^2\right\rangle \right\rbrace \bigcup \left\lbrace \left\langle a^{2^{i}}b^2 \right\rangle \mid  0\leq i\leq k  \right\rbrace \bigcup \left\lbrace \left\langle a^jb \right\rangle \mid  j~ \text{is even}  \right\rbrace \\
		 &\bigcup   \left\lbrace \left\langle a^jb^k \right\rangle \mid  j~ \text{is odd},~k=1~ \text{or}~3  \right\rbrace,
	\end{align*} where $o(a)=2^{k+1}t$, $o(a^{2^{k+1}}b^2)=2^kt$, $o(a^{2^{i}}b^2)=2^{k+2-i}t$ for $0\leq i\leq k$; $o(a^jb^k)=4$, if $j$ is even and $k=1$ or 3; $o(a^jb^k)=2$, if $j$ is odd and $k=1$ or 3. Also, $\left\langle a^ib\right\rangle \cap \left\langle a^jb\right\rangle =\left\langle b^2\right\rangle $. It follows that each $\sigma\in \mathscr{M}(V_{8n})$ fixes all the elements in $\left\lbrace \left\langle a^jb \right\rangle \mid  j~ \text{is even}  \right\rbrace$ and $\left\lbrace \left\langle a^jb^k \right\rangle \mid  j~ \text{is odd},~k=1~ \text{or}~3  \right\rbrace$. Notice that $o(a)= o(ab^2)$. However, the proper cyclic subgroups of $\left\langle a\right\rangle $ and $\left\langle ab^2\right\rangle $ are the same, so $\left\langle a\right\rangle $ and $\left\langle ab^2\right\rangle $ are fixed by each $\sigma \in \mathscr{M}(V_{8n})$. Also, for $1\leq i\leq k$, $\left\langle a^{2^i}b^2\right\rangle $ is the only maximal cyclic subgroup of order $2^{k+2-i}t$. It follows that $\left\langle a^{2^i}b^2\right\rangle $ is fixed by each  $\sigma \in \mathscr{M}(V_{8n})$. Hence $\mathscr{M}(V_{8n})=\{\textbf{1}_{M(V_{8n})}\}$.  
	
	Suppose $\mathit{Z}_n$ denotes the reduced power graph of a cyclic group of order $n$. Then the structure of $\mathcal{RP}(V_{8n})$ is the same as the graph shown in~\cite[Figure 5]{Ali}. From this structure, we have $$\mathscr{R}(V_{8n})= \widehat{a^2b}\cup \widehat{ab}\cup \widehat{a} \cup  \mathscr{R}\left( \left\langle a^2 \right\rangle \right) \cup \left[\mathscr{R}\left( \left\langle a^{2^{k}}b^2 \right\rangle  \right)\setminus \mathscr{R}\left( \left\langle a^{2^{k-1}}\right\rangle \right)\right] \displaystyle \bigcup_{i=1}^{k} \left[ \mathscr{R}\left( \left\langle a^{2^i}b^2 \right\rangle  \right)\setminus \mathscr{R}\left( \left\langle a^{2^{i+2}}\right\rangle \right)\right],$$ where $\widehat{a^2b}=\left\lbrace  a^jb^k \mid  j~ \text{is even},~k=1~ \text{or}~3  \right\rbrace$;  $\widehat{ab}=\left\lbrace  a^jb^k \mid  j~ \text{is odd},~k=1~ \text{or}~3  \right\rbrace$ and $\widehat{a}=[a]\cup[ab^2]$. Applying these in Theorems~\ref{auto} and~\ref{autRP}, we get 
	\begin{align*}
	Aut(\overrightarrow{\mathcal{RP}}(V_{8n}))&=Aut( \mathcal{RP}(V_{8n}))\\	&= S_{2n}\times S_{2n}\times S_{2n} \times \left( {\prod_{\widehat{u}\in \mathscr{R}(\left\langle a^2\right\rangle) }S_{\widehat{u}}} \right)\times \left({\prod_{\widehat{u}\in \mathscr{R}(\left\langle a^{2^{k}}b^2\right\rangle), \widehat{u}\notin \mathscr{R}(\left\langle a^{2^{k-1}}\right\rangle) }S_{\widehat{u}}}\right)\\ 
		&~~~\times \left(\displaystyle {\prod_{\widehat{u}\in \mathscr{R}(\left\langle a^{2^i}b^2\right\rangle), \widehat{u}\notin \mathscr{R}(\left\langle a^{2^{i+2}}\right\rangle)}S_{\widehat{u}}}\right) \\
		&=S_{2n}\times S_{2n}\times S_{2n} \times \left( {\prod_{\widehat{u}\in \mathscr{R}(\mathbb Z_{2^{k}t}) }S_{\widehat{u}}} \right)\times \left({\prod_{\widehat{u}\in \mathscr{R}(\mathbb Z_{2t}),\widehat{u}\notin \mathscr{R}(\mathbb Z_{t}) }S_{\widehat{u}}}\right)\\
		&~~ \times \left({\prod_{\widehat{u}\in \mathscr{R}(\mathbb Z_{2^{k}t}),\widehat{u}\notin \mathscr{R}(\mathbb Z_{t}) }S_{\widehat{u}}}\right)\\
		&=S_{2n}\times S_{2n}\times S_{2n} \times \left(\displaystyle\prod_{d\mid 2^{k}t}S_{\varphi(d)}\right)\times \left(\displaystyle\prod_{d\mid 2t, d\nmid t}S_{\varphi(d)}\right)\times \left( \displaystyle\prod_{d\mid 2^{k}t, d\nmid t}S_{\varphi(d)}\right).
	\end{align*}
	
	(2) Assume that $n=2$. 	Notice that	$$M(V_{8n})=\left\lbrace\left\langle a\right\rangle, \left\langle b\right\rangle,\left\langle ab\right\rangle , \left\langle a^2b\right\rangle, \left\langle a^3b\right\rangle, \left\langle ab^2\right\rangle, \left\langle ab^3, \right\rangle \left\langle a^2b^2\right\rangle, \left\langle a^3b^3\right\rangle \right\rbrace,$$
	 where $o(a) = o(b) = o(ab^2)=o(a^2b)=2$ and $o(ab)= o(a^3b)= o(ab^3)= o(a^2b^2)=o(a^2b^2)=2$. Also $\left\langle a\right\rangle \cap \left\langle ab^2\right\rangle =\left\langle a^2\right\rangle $ and $\left\langle b\right\rangle \cap \left\langle a^2b\right\rangle =\left\langle b^2\right\rangle $. It follows that $\mathscr{M}(V_{8n})\cong S_2$. Since $\mathscr{R}(V_{16})=\{\widehat{a},\widehat{b},\widehat{ab}\}$, where $\widehat{a}=[a]\cup[ab^2]$, $\widehat{b}=[b]\cup [a^2b]$ and $\widehat{ab}=[ab]\cup[a^3b]\cup[ab^3]\cup[a^2b^2]\cup[a^3b^3]$.  Applying these in Theorems~\ref{auto} and~\ref{autRP}, we get 
	\begin{align*}
		Aut( \overrightarrow{\mathcal{RP}}(V_{8n}))=Aut( \mathcal{RP}(V_{8n}))&= S_{\widehat{ab}}\times S_{\widehat{a}}\times S_{\widehat{b}} \ltimes S_2\\
		&=S_{5}\times (S_{4}\times S_4) \ltimes {S_2}\\
		&=S_{5}\times (S_{4}\wr S_{2}).
	\end{align*}
This completes the proof.
\end{proof}

%
	%
%

\section{Relation between $Aut\left(\mathcal{RP}(G)\right)$ (resp. $Aut(\overrightarrow{\mathcal{RP}}(G))$) and $Aut\left(\mathcal{P}(G)\right)$ (resp. $Aut(\overrightarrow{\mathcal{P}}(G))$)}
Let $G$ be a group. Consider the equivalence relation $\approx$ on $G$ defined as follows: for every $x,y\in G$, $x\approx y$ if and only if $N_{\mathcal{P}(G)}[x]=N_{\mathcal{P}(G)}[y]$. Let $\overline{x}$ denote the $\approx$-class determined by $x$ (cf.~\cite{cameron}).

\begin{thm}\label{Aut(RP)=Aut(P)} 
	Let $G$ be a finite group. Then we have the following:
	\begin{enumerate}[\normalfont(1)]
		\item $Aut\left(\mathcal{RP}(G)\right)\subseteq Aut({\mathcal{P}}(G))$ if and only if for every $x\in G$, $\widehat{x}$ is either of    Type~I or of Type~II with parameter $(2,r)$, where $r\geq 2$;
		
		\item  $Aut\left(\mathcal{P}(G)\right)\subseteq Aut\left(\mathcal{RP}(G)\right)$ if and only if  $\overline{x}=[x]$ for every $x\in G$.
	\end{enumerate}
\end{thm}

\begin{proof}
	(1) Suppose that for every $x\in G$, $\widehat{x}$ is of Type~II with parameter $(m,r)$, where $m\neq 2$. Then $\widehat{x}=\displaystyle\bigcup_{i=1}^{r}[x_i]$, where $o(x_i)=m$ for $i=1,2,\ldots, r$. Consider the function $\pi$ on $G$ defined by $x_1^\pi=x_2$, $x_2^\pi=x_1$ and $y^\pi=y$ for every $y\neq x_1,x_2$. Then $\pi$ is a bijection and $N_{\mathcal{RP}(G)}(x_1)=N_{\mathcal{RP}(G)}(x_2)$, so $\pi\in Aut\left(\mathcal{RP}(G)\right)$. Since $m\neq 2$, there exists $y\in G$ such that $\left\langle y\right\rangle =\left\langle x_1\right\rangle $ and $\left\langle y\right\rangle \neq \left\langle x_2\right\rangle $. So $(y,x_1)\in A(\mathcal{P}(G))$. However, $(y^\pi, x_1^\pi)=(y,x_2)\notin A(\mathcal{P}(G))$. Consequently, $Aut\left(\mathcal{RP}(G)\right)\nsubseteq Aut({\mathcal{P}}(G))$. Similarly, we can show that, if $\widehat{x}$ is of Type~III or Type~IV, then $Aut\left(\mathcal{RP}(G)\right)\nsubseteq Aut({\mathcal{P}}(G))$. 
	
	Next we assume that for every $x\in G$, $\widehat{x}$ is either of Type~I or of Type~II with parameter $(2,r)$, where $r\geq 2$. 
	Let $\pi\in Aut\left(\mathcal{RP}(G)\right)$ and let $x,y\in G$. Suppose that $x$ and $y$ are adjacent in $\mathcal{P}(G)$. Then $\left\langle x \right\rangle \subseteq \left\langle y\right\rangle$ or $\left\langle y\right\rangle \subseteq \left\langle x\right\rangle $. If $\left\langle x\right\rangle \neq \left\langle y\right\rangle $, then $x$ and $y$ are also adjacent in $\mathcal{RP}(G)$ and so $x^\pi$ and $y^\pi$ are adjacent in $\mathcal{RP}(G)$. It follows that $x^\pi$ and $y^\pi$ are adjacent in $\mathcal{P}(G)$. If $\left\langle x\right\rangle =\left\langle y\right\rangle $, then $\widehat{x}=\widehat{y}$. It turns out by assumption that  $ \widehat{x}^\pi = \widehat{y}^\pi $ and hence  $\widehat{x^\pi}=\widehat{y^\pi}$. By Proposition~\ref{order equal of eql class}(3), $\widehat{x^\pi}$ and $\widehat{y^\pi}$ are of Type~I. It follows that $[x^\pi]=[y^\pi]$ and so $\left\langle x^\pi\right\rangle =\left\langle y^\pi\right\rangle $. This implies that $x^\pi$ and $y^\pi$ are adjacent in $\mathcal{P}(G)$.

	Suppose $x$ and $y$ are not adjacent in $\mathcal{P}(G)$. Then $\left\langle x\right\rangle \nsubseteq \left\langle y\right\rangle $ and $\left\langle y\right\rangle \nsubseteq\left\langle x\right\rangle $. So $x$ and $y$ are also not adjacent in $\mathcal{RP}(G)$. It follows that $x^\pi$ and $y^\pi$ are not adjacent in $\mathcal{RP}(G)$. So $\left\langle x^\pi\right\rangle \not\subset\left\langle y^\pi\right\rangle $, $\left\langle y^\pi\right\rangle \not\subset \left\langle x^\pi\right\rangle $ or $\left\langle x^\pi\right\rangle = \left\langle y^\pi\right\rangle $. If $\left\langle x^\pi\right\rangle =\left\langle y^\pi\right\rangle $, then by a similar argument used above, we get $\left\langle x\right\rangle =\left\langle y\right\rangle $. So $x$ and $y$ are adjacent in $\mathcal{P}(G)$, which is a contradiction. Thus $\left\langle x^\pi\right\rangle \not\subset \left\langle y^\pi\right\rangle $ and $\left\langle y^\pi\right\rangle \not\subset \left\langle x^\pi\right\rangle $ and so $x^\pi$ and $y^\pi$ are not adjacent in $\mathcal{P}(G)$. Thus $\pi\in Aut\left(\mathcal{P}(G)\right)$ and so $Aut\left(\mathcal{RP}(G)\right)\subseteq Aut\left(\mathcal{P
	}(G)\right)$. 
	
	(2) Suppose there exists $x\in G$ such that $\overline{x}\neq [x]$. Then there exists $y\in \overline{x}$ such that $\left\langle x\right\rangle \neq \left\langle y\right\rangle $, and so $o(x)\neq o(y)$. Consider the function $\pi$ on $G$ defined by $x^\pi=y$, $y^\pi=x$ and $w^\pi=w$ for every $w\in G$ with $w\neq x, y$. Then for every $u,v\in G$, $(u,v)\in A(\mathcal{P}(G))$ if and only if $(u^\pi,v^\pi)\in A(\mathcal{P}(G))$. So $\pi\in Aut\left(\mathcal{P}(G)\right)$. Since $\overline{x}\neq [x]$, it follows that one of the following holds: (i) either $o(x)$ or $o(y)$ is a composite number or (ii) $x=e$ and $o(y)$ is a prime.
	
	Suppose that either $o(x)$ or $o(y)$ is a composite number. Without loss of generality, we assume that $o(x)$ is a composite number. Then by parts (2) and (3) of Proposition~\ref{order equal of eql class}, $\pi \notin Aut\left(\mathcal{RP}(G)\right)$, since  $o(x)\neq o(x^\pi)$. Thus $Aut\left(\mathcal{P}(G)\right)\nsubseteq Aut({\mathcal{RP}}(G))$. Suppose that $x=e$ and $o(y)$ is a prime. Then by Lemma~\ref{image of e}, $G\cong \mathbb Z_{2^m}$ $(m\geq 1)$ or $Q_{2^\alpha}$ $(\alpha\geq 3)$.  By Example~\ref{Aut(Zn)}(2) and \cite[Example 5.1(ii)]{Min Aut PG 2016}, it follows that  $Aut\left(\mathcal{P}(\mathbb Z_{2^m})\right)\nsubseteq Aut({\mathcal{RP}}(\mathbb Z_{2^m}))$. Now we show that $Aut\left(\mathcal{P}(Q_{2^\alpha})\right) \nsubseteq Aut({\mathcal{RP}}(Q_{2^\alpha}))$. Let $a^2, a^4 \in Q_{2^\alpha}$. Then $o(a^2)= 2^{\alpha +1}$ and $o(a^4)= 2^{\alpha}$. Consider the function $\pi$ on $Q_{2^\alpha}$ defined by $(a^2)^{\pi}=a^4$, $(a^4)^{\pi}=a^2$, and $x^{\pi}=x$ for all $x\neq a^2, a^4$. Since $N_{\mathcal{P}(G)}(a^2)=N_{\mathcal{P}(G)}(a^4)$, it follows that $\pi \in Aut\left(\mathcal{P}(Q_{2^\alpha})\right)$. But $deg_{\mathcal{RP}(Q_{2^\alpha})}(a^2)\neq deg_{\mathcal{RP}(Q_{2^\alpha}))}(a^4)$ and hence $\pi \notin Aut\left(\mathcal{RP}(Q_{2^\alpha})\right)$. Thus $Aut\left(\mathcal{P}(Q_{2^\alpha})\right) \nsubseteq Aut({\mathcal{RP}}(Q_{2^\alpha}))$.
	
	Assume that for every $x\in G$,  $\overline{x}=[x]$. Let $\pi\in Aut\left(\mathcal{P}(G)\right)$ and let $x,y\in G$. Suppose either $\overline{x}$ or $\overline{y}$ is $\overline{e}$. Then clearly, $\left\langle x\right\rangle \subset \left\langle y\right\rangle $ if and only if $\left\langle x^\pi\right\rangle \subset \left\langle y^\pi\right\rangle $. This implies that if $(x,y)\in A\left(\mathcal{RP}(G)\right)$, then $(x^{\pi},y^{\pi})\in A\left(\mathcal{RP}(G)\right)$. Now we assume that $\overline{x}, \overline{y}\neq \overline{e}$. If $(x,y)\in A\left(\mathcal{RP}(G)\right)$, then $(x,y)\in A\left(\mathcal{P}(G)\right)$ and so $(x^\pi,y^\pi)\in A\left(\mathcal{P}(G)\right)$. It follows that any one of the following holds: $\left\langle x^\pi\right\rangle  \subset \left\langle y^\pi\right\rangle $, $\left\langle y^\pi\right\rangle \subset \left\langle x^\pi\right\rangle $ or $\left\langle x^\pi\right\rangle =\left\langle y^\pi\right\rangle $. Suppose $\left\langle y^\pi\right\rangle = \left\langle x^\pi\right\rangle $. Then $\overline{x^\pi}=\overline{y^\pi}$ and so $\overline{x}=\overline{y}$. Hence $[x]=[y]$ and so $\left\langle x\right\rangle =\left\langle y\right\rangle $, which is a contradiction to the fact that $(x,y)\in A\left(\mathcal{RP}(G)\right)$. If  $\left\langle x^\pi\right\rangle  \subset \left\langle y^\pi\right\rangle $ or $\left\langle y^\pi\right\rangle  \subset \left\langle x^\pi\right\rangle $, then $(x^\pi,y^\pi)\in A\left(\mathcal{RP}(G)\right)$.
	
	Suppose $(x,y)\notin A\left(\mathcal{RP}(G)\right)$. Then either $\left\langle x\right\rangle \not \subset \left\langle y\right\rangle $ and $\left\langle y\right\rangle \not \subset \left\langle x\right\rangle $, or $\left\langle x\right\rangle =\left\langle y\right\rangle $. If $\left\langle x\right\rangle \not \subset \left\langle y\right\rangle $ and $\left\langle y\right\rangle \not  \subset \left\langle x\right\rangle $, then $(x,y)\notin A\left(\mathcal{P}(G)\right) $. This implies that $(x^\pi,y^\pi)\notin A\left(\mathcal{P}(G)\right)$ and so $(x^\pi,y^\pi)\notin A\left(\mathcal{RP}(G)\right)$, since $\mathcal{RP}(G)$ is a subgraph of $\mathcal{P}(G)$. If $\left\langle x\right\rangle =\left\langle y\right\rangle $, then $\overline{x}=\overline{y}$, and so $\overline{x^\pi}=\overline{y^\pi}$. It follows that $[x^\pi]=[y^\pi]$. Then $\left\langle x^\pi\right\rangle =\left\langle y^\pi\right\rangle $ and $(x^\pi,y^\pi)\notin A\left(\mathcal{RP}(G)\right)$. Thus $\pi\in Aut\left(\mathcal{RP}(G)\right)$ and so $Aut\left(\mathcal{P}(G)\right)\subseteq Aut({\mathcal{RP}}(G))$.
\end{proof}

The following result is an immediate consequence of Theorem~\ref{Aut(RP)=Aut(P)}.
\begin{cor}\label{RP=P}
	Let $G$ be a finite group. Then $Aut\left(\mathcal{RP}(G)\right)=Aut\left(\mathcal{P}(G)\right)$ if and only if for every $x\in G$, $\overline{x}=[x]$ and $\widehat{x}$ is either of Type~I or of Type~II with parameter $(2,r)$, where $r\geq 2$.
\end{cor}


\begin{thm}\label{DRP subset DP}
	Let $G$ be a finite group. Then we have the following:
	\begin{enumerate}[\normalfont(1)]
		\item  $Aut(\overrightarrow{\mathcal{P}}(G))\subseteq Aut(\overrightarrow{\mathcal{RP}}(G))$;
		
		\item $Aut( \overrightarrow{\mathcal{RP}}(G))\subseteq Aut(\overrightarrow{\mathcal{P}}(G))$ if and only if for every $x\in G$, $\widehat{x}$ is either of    Type~I or of Type~II with parameter $(2,r)$, where $r\geq 2$.
	\end{enumerate}
\end{thm}

\begin{proof}
	(1) Let $\pi\in Aut(\overrightarrow{\mathcal{P}}(G))$ and let $x,y \in G$. Then $\left\langle x\right\rangle =\left\langle y\right\rangle $ if and only if $\left\langle x^\pi\right\rangle =\left\langle y^\pi\right\rangle $; and $\left\langle y\right\rangle \subset \left\langle x\right\rangle $ if and only if $\left\langle y^\pi\right\rangle \subset \left\langle x^\pi\right\rangle $. It follows that $(x,y)\in A(\overrightarrow{\mathcal{RP}}(G))$ if and only if $(x^\pi,y^\pi)\in A(\overrightarrow{\mathcal{RP}}(G))$. Thus $\pi\in Aut( \overrightarrow{\mathcal{RP}}(G))$, and so $Aut(\overrightarrow{\mathcal{P}}(G))\subseteq Aut( \overrightarrow{\mathcal{RP}}(G))$.
	
	(2) Consider the function $\pi$ defined in the proof of Theorem~\ref{Aut(RP)=Aut(P)}(1). It is shown that $\pi \in Aut( \overrightarrow{\mathcal{RP}}(G))$. Hence, if there exists $x\in G$ such that $\widehat{x}$ is of Type~II with parameter $(m,r)$, where $m\neq 2$, or Type~III or Type~IV, then by a similar argument used in the proof of Theorem~\ref{Aut(RP)=Aut(P)}(1), we get $Aut( \overrightarrow{\mathcal{RP}}(G))\neq Aut(\overrightarrow{\mathcal{P}}(G))$. 
	
	Suppose that for every $x\in G$, $\widehat{x}$ is either of of Type~I or of Type~II with parameter $(2,r)$, where $r\geq 2$. 	 Let $\pi\in Aut(\overrightarrow{\mathcal{RP}}(G))$ and let $x,y\in G$. Then by Lemma~\ref{Auto presetve suset}(1), $\left\langle y\right\rangle \subset \left\langle x\right\rangle $ if and only if $\left\langle y^\pi\right\rangle \subset \left\langle x^\pi\right\rangle $. If $\left\langle x\right\rangle =\left\langle y\right\rangle $, then  $\widehat{x}=\widehat{y}$. It turns out by assumption that $\widehat{x}^\pi=\widehat{y}^\pi$. Hence  $\widehat{x^\pi}=\widehat{y^\pi}$. By Proposition~\ref{order equal of eql class}(3), $[x^\pi]=[y^\pi]$ and so $\left\langle x^\pi\right\rangle =\left\langle y^\pi\right\rangle $. Similarly, if $\left\langle x^\pi\right\rangle =\left\langle y^\pi\right\rangle $, then $\left\langle x\right\rangle =\left\langle y\right\rangle $. It follows that $\pi\in Aut(\overrightarrow{\mathcal{P}}(G))$. So $Aut( \overrightarrow{\mathcal{RP}}(G))\subseteq Aut(\overrightarrow{\mathcal{P}}(G))$.
\end{proof}

The following result is an immediate consequence of Theorem~\ref{DRP subset DP}.

\begin{cor}\label{DRP=DP}
	Let $G$ be a finite group. Then $Aut( \overrightarrow{\mathcal{RP}}(G))=Aut(\overrightarrow{\mathcal{P}}(G))$ if and only if for every $x\in G$, $\widehat{x}$ is either of  Type~I or of Type~II with parameter $(2,m)$, where $m\geq 1$.
\end{cor}

\begin{example}
	$Aut(\overrightarrow{\mathcal{RP}}(U_{6n}))= Aut\left( {\mathcal{RP}}(U_{6n})\right)=Aut\left( {\mathcal{P}}(U_{6n})\right)=Aut(\overrightarrow{\mathcal{P}}(U_{6n}))$.
\end{example}
\begin{proof}
	Notice that $Aut\left( {\mathcal{P}}(U_{6n})\right)$ is given in~\cite[Theorem 2.6]{Ali}. Suppose $\mathit{Z}_n$ denotes the reduced power graph of a cyclic group of order $n$. Then  the structure of $\mathcal{RP}(U_{6n})$ is the same as the graph shown  in~\cite[Figure 2]{Ali}. From the structure of $\mathcal{P}(U_{6n})$ and $\mathcal{RP}(U_{6n})$, we have $\overline{x}=\widehat{x}=[x]$. Hence by Corollaries~\ref{RP=P},~\ref{DRP=DP} and \cite[Proposition~5.2]{Min Aut PG 2016}, we get the result.
\end{proof}


%

%

%
%


\begin{thebibliography}{00}	
	
	\bibitem{survey}  Abawajy, J., Kelarev, A., Chowdhury, M. (2013). Power graphs: A survey. {\it Electron. J. Graph Theory Appl.} {1}:125--147. 
	
	\bibitem{Ali}  Ashrafi, A. R., Gholami, A., Mehranian, Z. (2017). Automorphism group of certain power graphs of finite groups. \emph{Electron. J. Graph Theory Appl.} {5}:70--82.
	
	\bibitem{AliD} Alireza, D., Ahmad, E., Abbas, J. (2015). Some results on the
	power graphs of finite groups. \emph{Sci. Asia}. {41}:73–78.
	
	\bibitem{raj RPG and PG 2018}  Anitha, T., Rajkumar, R. (2019). On the power graph and the reduced power graph of a finite group. \emph{Comm Algebra}. {47}:3329--3339.
	
	\bibitem{raj RPG embedd 2018}  Anitha, T., Rajkumar, R. (2020). Characterization of groups with planar, toroidal or projective planar (proper) reduced power graphs.  {\it J. Algebra Appl}. {19}: Article No. 2050099.
	
	\bibitem{bohanan} Bohanon, J. P., Les Reid. (2006). Finite groups with planar subgroup lattice. {\it Algebr Comb}. {23}:207--223.
	
	\bibitem{cameron} Cameron, P. J. (2010). The power graph of a finite group, II.  {\it J. Group Theory}. {13}:779--783.
	
	\bibitem{peter} Cameron, P. J., Ghosh, S. (2011). The power graph of a finite group, {\it Discrete Math}.   {311}:1220--1222.
	
	\bibitem{Chakrabarty}  Chakrabarty, I., Ghosh, S., Sen, M. K. (2009). Undirected power graphs of semigroups,  {\it Semigroup Forum}.  {78}:410--426.
	
	
	\bibitem{prime elt} Cheng Kai Nah., Deaconescu, M., Lang Mong Lung., Shi Wujie. (1993). Corrigendum and addendum to "Classification of finite groups with all elements of prime order", \emph{Proc. Amer. Math. Soc.} {117}:1205--1207.
	
	\bibitem{Curtin} Curtin, B.,  Pourgholi, G. R., Yousefi-Azari, H. (2015). On the punctured power graph of a finite group, \emph{Australas. J. Combin.} {62}:1--7.
	
	\bibitem{min feng 2015} Feng, M.,  Ma, X., Wang, K. (2015). The structure and metric dimension of the power graph of a finite group. {\it European J. Combin.} {43}:82--97.
	
	\bibitem{Min Aut PG 2016} Feng, M.,  Ma, X., Wang, K. (2016). The full automorphism group of the power (di)graph of a finite group. \emph{European J. Combin.} {52}:197--206.
	
	
	\bibitem{Li 2021} H. Li, M., Fu, R., Ma, X. (2021). Forbidden subgraphs in reduced power graphs of finite groups, \emph{AIMS Mathematics}. {6}:5410--5420.
	
	\bibitem{Kelarev 2000}  Kelarev, A. V., Quinn, S. J. (2000). A combinatorial property and power graphs of groups. \emph{Contrib.
	General Algebra}. 12:3--6.
	
	
	\bibitem{Kelarev 2001}  Kelarev, A. V., Quinn, S. J., Smol\'{i}kov\'{a}, R. (2021).  Power graphs and semigroups of matrices.
	{\it Bull. Austral. Math. Soc}. {63}:341--344. 
	
	\bibitem{Kelarev} Kelarev, A. V., Quinn, S. J. (2002). Directed graph and combinatorial properties of semigroups.  {\it J. Algebra}. {251}:16--26.
	
	
	
	\bibitem{Xmacode} Ma, X. (2020). Perfect codes in proper reduced power graphs of finite groups. {\it Comm. Algebra}. {48}:3881--3890.
	
	\bibitem{Xma strong 2021}	Ma, X., Zhai, L. (2021). Strong metric dimensions for power graphs of finite groups. {\it Comm. Algebra}. {49}:4577--4587.
	
	\bibitem{Xma 2021} Ma, X. (2021).	Proper connection of power graphs of finite groups.  {\it J. Algebra Appl}. {20}: Article No. 2150033 
	
	
	\bibitem{Ma genus 3} Ma, X., Li, L. (2022). On the metric dimension of the reduced power graph of a finite group. {\it Taiwan. J. Math. } {21}:1--15.
	
	
	\bibitem{Mehranian} Mehranian, Z., Gholami, A., Ashrafi, A. R. (2016).  A note on
	the power graph of a finite group. {\it Int. J. Group Theory}. {5}:1–10. 
	
	\bibitem{raj RPG 2017} Rajkumar, R., Anitha, T. (2017). Reduced power graph of a group. {\it Electron. Notes Discrete Math.} {63}:69--76.
	
	
	\bibitem{raj RPG 2018}  Rajkumar, R.,  Anitha, T. (2021). Some results on the reduced power graph of a group,  {\it Southeast Asian Bull. Math.} {45}:241–262.
	
	\bibitem{raj RPG LS 2019} Rajkumar, R., Anitha, T. (2021). Laplacian spectrum of reduced power graph of certain finite groups.  {\it Linear  Multilinear Algebra}.  {69}:1716--1733.
\end{thebibliography}
%
%

\end{document}